\documentclass[11pt]{amsart}
\usepackage{a4wide,enumerate,color,graphicx}
\usepackage[colorlinks, linkcolor=blue]{hyperref}
\usepackage{enumitem}
\usepackage[normalem]{ulem}
\usepackage{lineno}
\usepackage{amsmath}
\usepackage{comment}
\usepackage{appendix}
\usepackage{tikz}
\allowdisplaybreaks
\numberwithin{equation}{section}

\let\pa\partial  
\let\na\nabla  
\let\eps\varepsilon

\newcommand{\diver}{\operatorname{div}}

\newcommand{\dom}{{\mathcal O}}

\newcommand{\Prob}{\mathbb{P}}

\newcommand{\dt}{\textnormal{d}t}
\newcommand{\ds}{\textnormal{d}s}

\newcommand{\dr}{\textnormal{d}r}
\newcommand{\dW}{\textnormal{d}W}

\newcommand{\N}{{\mathbb N}}  
\newcommand{\R}{{\mathbb R}}

\definecolor{darkblue}{rgb}{0.1,0.1,0.9}
\newcommand{\hcommentg}[1]{\marginpar{\raggedright\scriptsize{\textcolor{darkblue}{#1}}}}
\definecolor{fh}{rgb}{0.54, 0.17, 0.89}

\definecolor{burntorange}{rgb}{0.8, 0.33, 0.0}

\newcommand{\red}{\textcolor{red}}

\newcommand\dela[2]{{\sout{#1}\,\color{blue}#2}}

\makeatletter
\def\namedlabel#1#2{\begingroup
    #2%
    \def\@currentlabel{#2}%
    \phantomsection\label{#1}\endgroup
}
\makeatother

\makeatletter
\newcommand{\mylabel}[2]{#2\def\@currentlabel{#2}\label{#1}}
\makeatother

\newtheorem{theorem}{Theorem}[section]   
\newtheorem{lemma}[theorem]{Lemma}   
   
\newtheorem{proposition}[theorem]{Proposition}   
\newtheorem{remark}{Remark}[section]
  
\newtheorem{definition}{Definition}[section]  
\newtheorem{example}{Example}[section] 
\newtheorem{assumption}{Assumption} 

\begin{document}  
\title[Pathwise mild and weak solutions for quasilinear SPDEs]{On the equivalence of pathwise mild and weak solutions for quasilinear SPDEs} 

\author[G. Dhariwal]{Gaurav Dhariwal}
\address{Institute for Analysis and Scientific Computing, Vienna University of  
	Technology, Wiedner Hauptstra\ss e 8--10, 1040 Wien, Austria}
\email{gaurav.dhariwal@tuwien.ac.at} 

\author[F. Huber]{Florian Huber}
\address{Institute for Analysis and Scientific Computing, Vienna University of  
	Technology, Wiedner Hauptstra\ss e 8--10, 1040 Wien, Austria}
\email{florian.huber@asc.tuwien.ac.at}

\author[A. Neam\c tu]{Alexandra Neam\c tu}
\address{Fakult\"at f\"ur Mathematik,
Universit\"at Bielefeld, 100131, D-33501 Bielefeld}
\email{alexandra.neamtu@uni-bielefeld.de} 
\date{\today}

\thanks{The first two authors acknowledge partial support from   
the Austrian Science Fund (FWF), grants I3401, P30000, W1245, and F65. The last author has been supported until December 2019 by a German Science Foundation (DFG) grant in the D-A-CH framework KU 3333/2-1.  AN also acknowledges support from the CRC 1283 “Taming uncertainty and profiting from randomness and low regularity in analysis,
stochastics and their applications”. The authors are grateful to Ansgar J\"ungel and Christian Kuehn for numerous discussions and suggestions.} 

\begin{abstract}
The main goal of this work is to relate weak and pathwise mild solutions for parabolic quasilinear stochastic partial differential equations (SPDEs).
Extending in a suitable way techniques from the theory of nonautonomous semilinear SPDEs to the quasilinear case, we prove the equivalence of these two solution concepts. 
\end{abstract}

\keywords{Quasilinear SPDEs, cross-diffusion systems, weak solution, pathwise mild solution.}  
 
\subjclass[2000]{60H15, 35R60, 35Q35, 35Q92.}

\maketitle


\section{Introduction}
The aim of this paper is to relate two solution concepts for quasilinear SPDEs, namely weak and pathwise mild,  with a particular emphasis on cross-diffusion {systems}. Such systems arise in numerous applications, for example, they can be used to describe the dynamics of interacting population species. A well-known model is the deterministic Shigesada-Kawasaki-Teramoto population system, which was introduced in \cite{SKT79} in order to analyze population segregation {between two species} by induced cross-diffusion. This system can be formally derived from a random-walk
model on lattices for transition rates which depend linearly on the population densities.
 Generalized population cross-diffusion models are obtained when the
dependence of the transition rates on the densities is nonlinear, see for instance~{\cite{ZaJu17}}.

In order to model population densities for $n\geq 2$ species, 
we consider cross-diffusion systems of the form
\begin{equation}
\begin{cases}\label{eq:system}
\textnormal{d}u=\diver\left(B(u)\nabla u\right)~\textnormal{d}t+\sigma(u)~\textnormal{d}W_t,\qquad t>0,\\
u(0)=u^0,
\end{cases}
\end{equation}
on an open, bounded domain $\dom\subset\R^d$ ($d\geq 1$), with smooth boundary $\partial \mathcal{O}$. Here $B=(B_{ij})$ is an $n\times n$ diffusion matrix, $\sigma$ is a nonlinear term and $W=(W_1,\ldots, W_n)$ is a cylindrical Wiener process. The precise assumptions on the coefficients will be stated in Section~\ref{sec:prelim} and~\ref{sec:main_thm}. 
The previous system rewrites componentwise as
\begin{equation}\label{1.eq}
  \textnormal{d}u_i - \diver\bigg(\sum_{j=1}^n B_{ij}(u)\na u_j\bigg)\textnormal{d}t
	= \sum_{j=1}^n\sigma_{ij}(u)\textnormal{d}W^j_t, \quad\ t>0,
\end{equation}
 with $u_i(0)=u_i^0\quad\mbox{in }\dom,\ i=1,\ldots,n$,
and is augmented by either no-flux boundary conditions
\begin{equation}\label{1.bic}
  \sum_{j=1}^n B_{ij}(u)\na u_j\cdot\nu = 0\quad\mbox{on }\pa\dom,\ i=1,\ldots,n,\ t>0,
\end{equation}
or homogeneous Dirichlet boundary conditions
\begin{equation}\label{1.bic_D}
  u_i(t,x) = 0\quad\mbox{on }\pa\dom,\ i=1,\ldots,n,\ t>0.
\end{equation}
The solution {$u_i:\Omega\times\dom\times[0,T]\to \mathbb{R}$} models the density of the $i^{\text{th}}$-population species at a current location $x\in\dom$ and a certain time $t>0$.\\

In order to investigate mild solutions for \eqref{eq:system}, we write it as an abstract quasilinear Cauchy problem
\begin{equation}\label{eq:1.Qspde}
    \begin{cases}
    \textnormal{d}u = A_u u \,\dt + \sigma(u)\,\dW_t\\
    u(0)=u^0,
    \end{cases}
\end{equation}
where the linear operator $A_u$ is given by $A_u v := \diver(B(u) \nabla v)$.\\

Due to their numerous applications, quasilinear SPDEs have attracted considerable interest (e.g. \cite{BaPrRo09, Gess12, LiRo10, PrRo07,HoZh17, DHV16, DMH15}) which has also broadened the scope of available solution concepts such as kinetic\,\cite{DHV16, FeGe19, GeHo18}, entropy\,\cite{DaGeGe19}, martingale\,\cite{DHV16,DJZ19,DHJKN19}.
Numerous developments for quasilinear SPDEs have been recently {made} in the context of rough paths theory\,\cite{OtWe19}, paracontrolled calculus\,\cite{BaDeHo19, FuGu19}, or regularity structures\,\cite{GeHa19}.

Another solution concept, so-called pathwise mild solution, for semilinear parabolic SPDEs with nonautonomous random generators was introduced by Pronk and Veraar in \cite{PrVe15}, where they bypassed the issue of non adaptive integrand in the definition of the It\^o integral by the use of integration by parts. This solution concept was then extended to the case of quasilinear parabolic SPDEs, including stochastic SKT system, by Kuehn and the last named author \cite{KuNe18} where they proved the existence of a unique local-in-time pathwise mild solution for the equations of the form \eqref{eq:1.Qspde}.
A stochastic process $u$ is called a pathwise mild solution of \eqref{eq:1.Qspde} if
\begin{equation}
\label{eqn-sol-path-mild}
    u(t)=U^u(t,0)u^{0} - \int_{0}^{t} U^u(t, s) A_u(s) \int_{s}^{t}\sigma(u(\tau))\dW_{\tau}\, \ds + U^u(t, 0) \int_{0}^{t}\sigma(u(s))\dW_{s},
\end{equation}
where $U^u(\cdot,\cdot)$, is the random evolution family generated by $A_u$, see Section~\ref{sec:prelim} for further details. This formula can be motivated using integration by parts and overcomes the non-adaptedness of the random evolution family $U^u(\cdot,\cdot)$ required in order to define the stochastic convolution as an It\^o integral.
Pathwise mild solutions for hyperbolic SPDEs with additive noise were analyzed in~\cite{MS17}.

The non-adaptedness of the integrand in the definition of a  stochastic integral was firstly discussed by Al\'os, Le\'on and Nualart in \cite{AlLeNu99, LeNu98} using the Skorokhod- and the Russo-Vallois~\cite{RuVa93} forward integral. Similar to~\cite{PrVe15}, in~\cite{LeNu98} such problems arise for semilinear SPDEs with random, nonautonomous generators.
Furthermore, in \cite{AlLeNu99, LeNu98}  it {was shown} that a Skorokhod-mild solution for such SPDEs does not satisfy the weak formulation, whereas the forward mild (based on the Russo-Vallois integral) does. In
\cite{PrVe15}, the authors showed that the pathwise mild solution is equivalent to the forward mild one.

{The} equivalence results for weak, pathwise- and forward-mild solutions for semilinear SPDEs obtained in~\cite{PrVe15} and the existence of pathwise mild solutions for quasilinear SPDEs obtained in~\cite{KuNe18}, {raise} a natural question {regarding} the equivalence of these solution concepts in the quasilinear case.
 Such an aspect is also important to study from a numerical point of view, since numerical schemes are aligned to the solution concept at hand. The same holds true for dynamical systems. In fact, many results regarding dynamics and asymptotic behavior of semilinear SPDEs, rely on a semigroup approach. If we take the SKT system as a motivation, then there are deterministic results
regarding the existence of attractors using weak~\cite{PhTe17} as well as mild~\cite{Ya08} solution concepts.

We emphasize that for semilinear PDEs and SPDEs numerous results regarding the equivalence of weak and mild solutions are well-known, see e.g.~\cite{Am93,Am95} and \cite{Ba77} for PDEs and e.g.~\cite{DPZ92} and~\cite{JVM08} for SPDEs.
 On the other hand, for quasilinear PDEs and SPDEs the literature discussing the equivalence of various solution concepts is very scarce. Therefore, we contribute to this aspect and establish the equivalence between pathwise mild and weak solutions for quasilinear SPDEs of the form \eqref{1.eq}, {see Theorem~\ref{thm:equivalence 1} and Theorem~\ref{thm:equivalence 2}. {Results regarding the existence of strong solutions for quasilinear PDEs are available in~\cite{Am86}, respectively for SPDEs in~\cite{Hornung}.  For certain elliptic-parabolic PDEs, using accretive operators and nonlinear semigroups, assertions regarding the equivalence of weak and mild solutions have been derived in \cite{GRW18,OuTo02} and the references specified therein. For nonlinear degenerate problems, the concept of entropy solution was introduced by Carrillo~\cite{Carrillo99}. In~\cite{Kobayasi03}, the equivalence between weak and entropy solutions was established for an elliptic-parabolic-hyperbolic degenerate PDE.
 However, to the best of our knowledge there are no other works in the literature discussing the equivalence of various solution concepts for quasilinear SPDEs such as \eqref{eq:1.Qspde}.}
 
The outline of our paper is as follows. In Section~\ref{sec:prelim}, we {introduce basic notations and collect} results from
the theory of evolution families generated by nonautonomous, random sectorial operators. These are necessary in order to introduce the concept of a pathwise mild solution for \eqref{eq:1.Qspde}. Section~\ref{sec:main_thm} contains our main result, which establishes under suitable assumptions on the coefficients, the equivalence of pathwise mild and weak solutions for~\eqref{eq:1.Qspde}.
The main idea is to approach the quasilinear SPDE \eqref{eq:1.Qspde} as a semilinear SPDE where the nonlinear map $A_u$ is viewed as a linear nonautonomous map for a fixed $u$, which is the pathwise mild solution of \eqref{eq:1.Qspde}. Thereafter we employ similar tools as in \cite{PrVe15, KuNe18} to prove the equivalence of weak and pathwise mild solutions for \eqref{eq:1.Qspde}.
Finally, we provide in Section~\ref{sec:example} examples of quasilinear SPDEs, to which the theory developed in this paper applies.  These include the stochastic SKT system. 


\section{Preliminaries}
\label{sec:prelim}
Let $T > 0$ be arbitrary but fixed and  $(\Omega,\mathcal{F},(\mathcal{F})_{t\in[0,T]},\Prob)$ be a filtered stochastic basis.
Let $X$ ($\|\cdot\|_X, \langle \cdot, \cdot\rangle_X$), $Y$ ($\|\cdot\|_Y, \langle\cdot, \cdot\rangle_Y$) and $Z$ ($\|\cdot\|_Z, \langle\cdot, \cdot\rangle_Z$) be separable Hilbert spaces such that the embeddings
\[
Z \hookrightarrow Y \hookrightarrow X
\]
are continuous. The choice of these Hilbert spaces depends on the corresponding quasilinear problem, see Section~\ref{sec:example} for concrete examples. We identify $X$ with its topological dual $X^\ast$.
Let $H$ denote another separable Hilbert space with orthonormal basis $\left\{\eta_{n}\right\}_{n\in\mathbb{N}}$ and $(W_{t})_{t\in[0,T]}$ is a cylindrical Wiener process over $X$, taking values in $H$. The cylindrical Wiener process $(W_t)_{t\geq 0}$ can be written as the series
\begin{equation}
\label{eq:wiener_process}
W_{t}=\sum_{n=1}^{\infty}e_{n}\beta_{n}(t),
\end{equation}
where $\left\{\beta_{n}(\cdot)\right\}_{n\in \mathbb{N}}$ are mutually independent real valued standard Brownian motions and $\{e_{n}\}_{n\in\mathbb{N}}$ denotes an orthonormal basis of $X$ and the sequence~\eqref{eq:wiener_process} converges in $H$ $\mathbb{P}$-a.s. The space of Hilbert-Schmidt operators from $H$ to $X$ will be denoted by $\mathcal{L}_2(H; X)$ and will be endowed with the norm
$$
\|L\|_{\mathcal{L}_{2}\left(H ; X\right)}^2 := \sum_{k=1}^{\infty}\left\|L \eta_{k}\right\|^2_{X}.
$$
We recall some auxiliary results related to the regularity of the stochastic integral with respect to a cylindrical Wiener process.

\begin{proposition} \cite[Prop.~4.4]{PrVe15F}
\label{prop:reg_stoc_int_1}
Let $p \in[2, \infty)$, $0 < \alpha < 1/2$ and $\sigma$ be a strongly measurable adapted process belonging to $L^{0}\left(\Omega ; L^{p}(0, T ;\mathcal{L}_{2}(H, X))\right)$. Then, the stochastic integral
\[
\int_{0}^{\cdot}\sigma(r)\,\dW_r\in L^{0}(\Omega;W^{\alpha,p}(0,T;X)).
\]
\end{proposition}

The continuous embedding $W^{\alpha,p}(0,T;X)\hookrightarrow C^{\alpha - 1/p} (0,T;X)$, for $1/p < \alpha < 1/2$, provides H\"older regularity of the stochastic integral, see \cite[Prop.~4.1]{PrVe15} for the full generality of the statement.

\begin{proposition} \label{prop:reg_stoc_int_2}
 Let $p \in[2, \infty)$, $1/p < \alpha < 1/2$ and $\sigma \in L^{0}\left(\Omega ; L^{p}(0, T ;\mathcal{L}_{2}(H, X))\right)$ be a strongly measurable adapted process. Then, there exists a positive constant $C_T$ independent of $\sigma$, converging to $0$ for $T \searrow 0$, such that
\begin{equation*}
    \left\|\int_{0}^{t}\sigma(r)~\dW_r\right\|_{L^{p}(\Omega,C^{\alpha - 1/p}(0,T;X))}\leq C_{T}\left\|\sigma\right\|_{L^{0}\left(\Omega ; L^{p}(0, T ;\mathcal{L}_{2}(H, X))\right)}.
\end{equation*}
\end{proposition}

We let $\mu > 0$ and introduce the following function space
\[
\mathcal{Z} := L^{0}\left(\Omega;L^{\infty}([0, T] ; Z) \cap C^{\mu}([0, T] ; Y)\right).
\]
Throughout this section and in Section~\ref{sec:main_thm}, we choose $u$, $v \in \mathcal{Z}$ be $(\mathcal{F}_t)_{t\in[0,T]}$-adapted stochastic processes.

Recalling \eqref{eq:1.Qspde}, we write 
\begin{equation}
\label{eq:operator}
    u\mapsto A_{v}u=\diver\left(B(v)\nabla u\right).
\end{equation}

Therefore, $A_v$ is a linear time-dependent random operator. In order to highlight this dependence we use the notation 
\begin{equation*}
    A_{v}(t,\omega):=A_{v(t,\omega)}.
\end{equation*}
To simplify the notation, we drop the parameter $\omega$  and simply write $A_v(t)$.
We now collect essential results regarding evolution families. These are extracted from \cite{PrVe15}. For further details regarding evolution systems for nonautonomous operators, we refer to the monographs by Pazy \cite{Pa83} and Yagi \cite{Ya10} as well as to \cite{AT92}.

In order to deal with the time and $\omega$-dependence of $A_{v}$ we impose, as in~\cite{PrVe15}, that the Acquistapace-Terreni conditions hold for every $\omega\in\Omega$. These were introduced in~\cite{AT92} for time-dependent generators and involve a sectoriality condition on $A_v$ together with a suitable H\"older-regularity. 
Since we are dealing with nonlinear generators we additionally impose a certain Lipschitz continuity {assumption}~\cite{Ya10,KuNe18}.

\begin{assumption}\textnormal{(Assumptions generators)}\label{ass:AT}
For $\vartheta\in (\frac{\pi}{2},\pi)$, let $\Sigma_{\vartheta}$ be an open sectorial domain, i.e.
\begin{equation}
\label{eq:sec_dom}
    \Sigma_\vartheta := \{\lambda \in \mathbb{C}:|\arg \lambda|<\vartheta\}.
\end{equation}
\begin{enumerate}
    \item [\mylabel{ass:AT1}{$\mathrm{(A1)}$}] $A_v$ is a sectorial operator on $X$, i.e., there exists a $\vartheta \in (\frac{\pi}{2},\pi)$, such that {for every $(t,\omega)\in [0,T]\times \Omega$, }
\[
\Sigma_{\vartheta} \cup\{0\} \subset \rho(A_{v}(t, \omega)).
\]
    
    \item [\mylabel{ass:AT2}{$\mathrm{(A2)}$}] The resolvent operator $\left(\lambda \mathrm{Id} - A_v\right)^{-1}$ satisfies the Hile-Yosida condition, i.e., there exists a constant $M \geq 1$ such that {for every $(t,\omega)\in [0,T]\times \Omega$, }
\begin{equation}
    \left\|(\lambda \mathrm{Id}-A_{v}(t,\omega))^{-1}\right\|_{\mathcal{L}(X)} \leq \frac{M}{|\lambda|+1},\; \text { for } \lambda \in \rho(A_{v}(t,\omega)).\nonumber
\end{equation}
  \item [\mylabel{ass:AT3}{$\mathrm{(A3)}$}]There exist two exponents $\nu, \delta \in (0,1]$ with $\nu+\delta>1$ such that {for every $\omega \in \Omega$} there exists a constant $L(\omega)\geq 0$ such that for all $s, t \in [0,T]$ 
\begin{equation}
    \left\|(A_{v}(t, \omega))^{\nu}\left(A_{v}(t, \omega)^{-1}-A_{v}(s, \omega)^{-1}\right)\right\|_{\mathcal{L}(X)} \leq L(\omega)|t-s|^{\delta}.\nonumber
\end{equation}
 \item [\mylabel{ass:AT4}{$\mathrm{(A4)}$}]Let $0<\nu \leq 1$ be fixed. Then, {for every $\omega \in \Omega$} there exists a constant $L(\omega) > 0$ such that
\begin{equation} 
    \left\|(A_{u}(t,\omega))^{\nu}\left(A_{u}(t,\omega)^{-1}-A_{v}(t,\omega)^{-1}\right)\right\|_{\mathcal{L}(X)} \leq L({\omega})\|u(t,\omega)-v(t,\omega)\|_{Y}\nonumber.
\end{equation}
\end{enumerate}
\end{assumption}
The conditions \ref{ass:AT1}--\ref{ass:AT3} will be referred to as the (AT) conditions. \\

Since we aim to relate mild and weak solutions for~\eqref{eq:1.Qspde}, we impose similar assumptions on the adjoint $A^\ast_v$ defined on $X^\ast$ of the operator $A_v$ with parameters $\vartheta^\ast, M^\ast, \nu^\ast$ and $\delta^\ast$. {Recall that $X$ is a Hilbert space and we identified it with its dual $X^\ast$.}

\begin{assumption}\textnormal{(Assumptions adjoint operators)}
\label{ass:ATadjoint}
\begin{enumerate}
\item [\mylabel{ass:AT1ad}{$\mathrm{(A1^*)}$}] $A_v^*$ is a sectorial operator, i.e., there exists a $\vartheta^\ast \in (\frac{\pi}{2},\pi)$, such that {for all $(t,\omega)\in [0,T]\times \Omega$,}
\[
\Sigma_{\vartheta^{*}} \cup\{0\} \subset \rho(A_{v}^{*}(t, \omega)),
\]
{where $\Sigma_{\vartheta^\ast}$ is defined as in \eqref{eq:sec_dom}.}    
    
    \item [\mylabel{ass:AT2ad}{$\mathrm{(A2^*)}$}] There exists a constant $M^{*} \geq 1$ such that for every $(t,\omega) \in [0,T] \times \Omega$,
\begin{equation}
    \left\|(\lambda \mathrm{Id}-A_{v}^{*}(t,\omega))^{-1}\right\|_{\mathcal{L}(X)} \leq \frac{M^{*}}{|\lambda|+1},\; \text { for } \lambda \in \rho(A_{v}^{*}(t,\omega)).\nonumber
\end{equation}
  \item [\mylabel{ass:AT3ad}{$\mathrm{(A3^*)}$}]There exist two exponents $\nu^{*}, \delta^{*} \in (0,1]$ with $\nu^{*}+\delta^{*}>1$ such that {for  every $\omega \in \Omega$} there exists a constant $L^*(\omega)\geq 0$ such that for all $s,t \in [0,T]$
\begin{equation}
    \left\|(A_{v}^{*}(t, \omega))^{\nu^*}\left(A_{v}^{*}(t, \omega)^{-1}-A_{v}^{*}(s, \omega)^{-1}\right)\right\|_{\mathcal{L}(X)} \leq L^{*}(\omega)|t-s|^{\delta^{*}}.\nonumber
\end{equation}
 \item [\mylabel{ass:AT4ad}{$\mathrm{(A4^*)}$}]Let $0<\nu^{*} \leq 1$ be fixed. Then, for every $\omega \in \Omega$ there exists a constant $L^{*}(\omega) > 0$ such that
\begin{equation} 
    \left\|(A^{*}_{u}(t,\omega))^{\nu^*}\left(A_{u}^{*}(t,\omega)^{-1}-A_{v}^{*}(t,\omega)^{-1}\right)\right\|_{\mathcal{L}(X)} \leq L^{*}({\omega})\|u(t,\omega)-v(t,\omega)\|_{Y}\nonumber.
\end{equation}
\end{enumerate}
\end{assumption}

{We refer the reader to Appendix~\ref{sec:frac_pow_dom} for details on the fractional power of the operator $A_u$ and its adjoint.}

\begin{remark}
\label{rem:dis_ass}
\begin{enumerate}
    \item [\textnormal{(1)}] We assume that the random variables $L(\omega)$, $L^{*}(\omega)$ are uniformly bounded with respect to $\omega$. This assumption can be dropped by a suitable localization argument, see \cite[Section 5.3]{PrVe15}.

\item [\textnormal{(2)}] The assumption \ref{ass:AT1} implies that $-A_{v}$ is a sectorial operator. Alternatively, one can assume as in \cite{Ya10}, that
    the spectrum of  $A_{v}(t,\omega)$ is contained in an open sectorial domain with angle $0<\varphi<\frac{\pi}{2}$ for every $t\in[0,T]$ and $\omega\in\Omega$, i.e.
\begin{equation}\label{ass:sectorial}
    \sigma(A_{v}(t,\omega))\subset \Sigma_{\varphi}:=\{\lambda \in \mathbb{C}:|\arg \lambda|<\varphi\},\nonumber
\end{equation}
which would imply that $A_v$ is a sectorial operator.

\item [\textnormal{(3)}] {Examples of operators satisfying \textbf{Assumptions~\ref{ass:AT} and \ref{ass:ATadjoint}} are given in Section~\ref{sec:example}.}
\end{enumerate}
\end{remark}
 
\begin{assumption}\textnormal{(Constant domains)}
\label{ass:const_dom}
{For simplicity we assume that the domains of $A_v$ and $A^{\ast}_v$ are constant, i.e. \ref{ass:AT3} and  \ref{ass:AT3ad} are satisfied for $\nu=1$ and $\nu^\ast = 1$, respectively. These conditions are called in literature the Kato--Tanabe assumptions \cite{AT92,PrVe15}.}
\end{assumption}

\begin{definition}
\label{defn:dom}
To emphasize the fact that we are working with constant domains, we introduce the following notations for $t \in [0,T]$ and $\omega \in \Omega$$\colon$
\begin{align*}
 \mathcal{D}_{v}& := D(-A_{v}(t,\omega)), \qquad \qquad \quad \;\; \mathcal{D}_{v}^{\alpha}  := D\left((-A_{v}(t,\omega))^{\alpha}\right),\\
 \mathcal{D}_{v}^{*}& := D\left((-A_{v}(t,\omega))^{*}\right), \qquad \qquad \mathcal{D}_{v}^{\alpha\ast}:=D\left(\left(\left(-A_{v}(t,\omega)\right)^{\alpha}\right)^\ast\right).
\end{align*}
\end{definition}

In general, conditions \ref{ass:AT1} and \ref{ass:AT2} can be difficult to verify for a given system, but in the Hilbert space framework, it suffices to apply the following criterion. According to \cite[Chapter~2.1]{Ya10}, we can associate to $-A_v$ a bilinear form $a(v; \cdot, \cdot)$ on a separable Hilbert space $V$ which is densely and continuously embedded in $X$. More precisely, we set
\begin{equation}
\label{eq:sesquilinear}
a(v; {w_1}, {w_2}) := \langle -A_v {w_1}, {w_2} \rangle_X,\qquad \mbox{ for all }\, {w_1}, {w_2} \in V.
\end{equation}
In this case, in order to verify \ref{ass:AT1} and \ref{ass:AT2} it suffices to show that  \cite[Chapter~2.1.1]{Ya10}
\begin{align}
    &a(v; {w},{w})\geq \kappa \|{w}\|_{V}^{2},\qquad \qquad \;\quad \;\; \quad \forall\, {w} \in V,\label{blf:coercivity}\\
    &\vert a(v; {w_1}, {w_2})\vert\leq M\|{w_1}\|_{V}\|{w_2}\|_{V},\qquad \forall\, {w_1}, {w_2}\in V,\label{blf:continuity}
\end{align}
for some constants $\kappa >0$ and $M> 0$. 

{The assumptions} \ref{ass:AT1}--\ref{ass:AT3} {allow us to} apply pathwise the deterministic results {from}~\cite{AT92} for the generation of an evolution family for nonautonomous operators and obtain as in~\cite[Theorem 2.2]{PrVe15} the following statement.

\begin{theorem}
\label{thm:semigp}
Let $\varDelta:=\{(s,t)\in [0,T]^{2}\,:\,s\leq t\}$. Assume that \ref{ass:AT1}--\ref{ass:AT3} hold true for the linear {nonautonomous} operator $A_{v}(t,\omega)$. Then, there exists a unique map $U^{v}\,:\,\varDelta\times\Omega\xrightarrow{} \mathcal{L}(X)$ such that
\begin{itemize}
    \item[\textnormal{(T1)}] for all $t \in[0, T]$, $U^{v}(t, t) = {\mathrm{Id}}$;
    \item[\textnormal{(T2)}] for all $r \leq s \leq t$, $U^{v}(t, s) U^{v}(s, r) = U^{v}(t, r)$;
    \item[\textnormal{(T3)}] for every $\omega \in \Omega$, the map $U(\cdot, \cdot, \omega) \colon \varDelta \to \mathcal{L}(X)$ is strongly continuous;
    \item[\textnormal{(T4)}] there exists a mapping $C: \Omega \rightarrow \mathbb{R}_{+},$ such that for all $s \leq t,$ one has
$$
\|U^{v}(t, s)\|_{\mathcal{L}(X)} \leq C;
$$
    \item[\textnormal{(T5)}] for every $s<t$ it holds pointwise in $\Omega$ that 
    \begin{equation}
    \label{eq:dt-semigroup}
       \frac{\partial}{\partial t} U^{v}(t, s)=A_{v}(t) U^{v}(t, s). 
    \end{equation}
    Moreover, there exists a mapping $C: \Omega \rightarrow \mathbb{R}_{+}$ such that
$$
\|A_{v}(t) U^{v}(t, s)\|_{\mathcal{L}(X)} \leq C(t-s)^{-1}.
$$
\end{itemize}
\end{theorem}

In \cite[Prop.~2.4]{PrVe15} the following measurability result for the evolution family $U^v$ was established. This fact prevents us from defining the stochastic convolution as an It\^o-integral.

\begin{proposition}
\label{prop:meas}
 The evolution system $U^{v}: \varDelta \times \Omega \rightarrow \mathcal{L}(X)$ is strongly measurable in the uniform operator topology. Moreover, for each $t \geq s,$ the mapping $\omega \mapsto U^{v}(t, s, \omega) \in \mathcal{L}(X)$ is strongly $\mathcal{F}_{t}$-measurable in the uniform operator topology.
\end{proposition}

\begin{remark}
If one replaces~\ref{ass:AT1} by \eqref{ass:sectorial}, then \eqref{eq:dt-semigroup} becomes
\[\frac{\partial}{\partial t} U^{v}(t, s) = - A_{v}(t) U^{v}(t, s).\]
\end{remark}

In the following, we point out spatial- and time-regularity results of the evolution family $U^{v}$, \textit{cf.} \cite[Lemma~2.6]{PrVe15}. For the convenience of the reader, we provide a brief overview on fractional powers of sectorial operators in Appendix~\ref{sec:frac_pow_dom}.

\begin{lemma}\label{lem:Semigroup estimates} Let conditions \ref{ass:AT1}--\ref{ass:AT3} be satisfied by  the linear operator $A_{v}(t,\omega)$. Then, there exists a mapping $C\colon \Omega \rightarrow \R_{+}$ such that for all $\,0\leq s<t\leq T$, $\theta\in [0,1]$, $\lambda\in(0,1)$ and $\gamma \in [0,\delta)$, the following estimates are valid
\begin{align}
&\left\|U^{v}(t,s)\left(-A_{v}(s)\right)^{\gamma}x\right\|_{X}\leq C \dfrac{\left\|x\right\|_{X}}{(t-s)^{\gamma}},\qquad x\in \mathcal{D}^{\gamma}_{v},\label{eq:UAg}\\ &\left\|(-A_{v}(t))^{\theta}U^{v}(t,s)(-A_{v}(s))^{-\theta}\right\|_{\mathcal{L}(X)}\leq C \label{eq:AgUA-g}.
\end{align}
Moreover, for $(s,t)\in \varDelta$, the map
\[ 
(s,t)\mapsto (-A_{v}(t))^{\theta}U^{v}(t,s)(-A_{v}(s))^{-\theta}
\]
is strongly continuous.
\end{lemma}
The following result, see \cite[Lemma~2.7]{PrVe15} for the proof, allows one to improve the regularity of the evolution family $U^v$, provided that the adjoint operator $A^\ast_v$ satisfies the assumptions \ref{ass:AT1ad}--\ref{ass:AT3ad}. 

\begin{lemma}\label{lem:Adjoint estimates}  
Let $A_{v}(t,\omega)$ and $A^\ast_v(t,\omega)$ satisfy \ref{ass:AT1}--\ref{ass:AT3} and \ref{ass:AT1ad}--\ref{ass:AT3ad}, respectively. Then, for every $t\in(0,T]$, the map $s\mapsto U^{v}(t,s)$ belongs to $C^{1}\left([0,t),\mathcal{L}(X)\right)$, and for every $x\in \mathcal{D}_{v}$ and  $\omega\in\Omega$, it holds that
\begin{equation}
\label{eq:ds-semigroup}
\dfrac{\partial}{\partial s}U^{v}(t,s)x = -U^{v}(t,s)A_{v}(s)x. 
\end{equation} 
Moreover, for $\beta\in [0,1]$, $0<\gamma<\delta^{*}$, $0 \leq \theta< \delta^{*}$, $\mu \in(0,1)$, $\lambda\in(0,1)$, the following inequalities hold:
\begin{align}
    &\left\|U^{v}(t, s)(-A_{v}(s))^{\beta} x\right\|_{X} \leq C \frac{\|x\|_{X}}{(t-s)^{\beta}},\qquad x \in \mathcal{D}_{v}^{\beta},\label{eq:UAb} \\
    &\left\|U^{v}(t, s)(-A_{v}(s))^{1+\theta} x\right\|_{X} \leq C \frac{\|x\|_{X}}{(t-s)^{1+\theta}},\qquad x \in \mathcal{D}_{v}^{1+\theta},\label{eq:UA1+b}\\
    &\left\|(-A_{v}(t))^{-\lambda} U^{v}(t, s)(-A_{v}(s))^{1+\gamma} x\right\|_{X} \leq C \frac{\|x\|_{X}}{(t-s)^{1+\gamma-\lambda}}, \quad x \in \mathcal{D}_{v}^{1+\gamma}\label{eq:A-gUA1+b}.
\end{align}
\end{lemma}

\begin{remark}
\label{rem:exponents}
Note that the estimates~\eqref{eq:UA1+b} and~\eqref{eq:A-gUA1+b} hold true only if the adjoint operator $(A_v(t,\omega))^{\ast}$ satisfies~\ref{ass:AT1ad}--\ref{ass:AT3ad}.
\end{remark}

\begin{remark}\label{rem:op_symm}
Lemma~\ref{lem:Semigroup estimates} and Lemma~\ref{lem:Adjoint estimates} remain valid if the roles of $A_{v}(t)$ and $(A_{v}(t))^{*}$ are interchanged. This is due to the fact that we identified $X$ with its dual $X^{*}$.
\end{remark}
{
For the sake of completeness, we point out the following statement on the adjoint of an evolution family $U^v$. Regarding {\bf Assumption \ref{ass:ATadjoint}}  we conclude by Theorem~\ref{thm:semigp} that for every $(t,\omega) \in (0,T]\times\Omega$ the nonautonomous linear operators $(A^\ast_v(t-\tau, \omega))_{\tau\in[0,t]}$ generate an evolution family $V^v(t;\tau,s)_{0\leq s \leq \tau \leq t}$. Due to~\cite[Proposition 2.9]{AFT91}  $$(U^v(t,s))^{\ast} = V(t;t-s,0) \quad \mbox{ for } s,t\in \Delta.$$
}

Based on Lemma~\ref{lem:Adjoint estimates}, we motivate a very useful identity resembling the fundamental theorem of calculus. This will be extensively used in Section~\ref{sec:main_thm}. For further details we refer to \cite{PrVe15}.

\begin{lemma}\label{lem:fun_cal}
Let \ref{ass:AT1}--\ref{ass:AT3} and \ref{ass:AT1ad}--\ref{ass:AT3ad} be satisfied by $A_{v}(t,\omega)$ and $A^\ast_{v}(t,\omega)$, respectively. For $\lambda>\delta^{*}$, $x^{*} \in \mathcal{D}_{v}^{\lambda*}$, the following identity holds true for every $\omega \in \Omega$ and $x\in X$:
\begin{equation}
       \label{identity}
       \int_{0}^{t}\left\langle U^{v}(t, s) A_{v}(s) x, x^{*}\right\rangle\ds=\left\langle U^{v}(t, 0) x, x^{*}\right\rangle - \left\langle x, x^{*}\right\rangle,
\end{equation}
{where $\langle \cdot , \cdot \rangle$ denotes the inner product in $X$.}
\end{lemma}

\begin{proof}[Sketch of the proof]
{Since $A_v(t,\omega)$ satisfies \ref{ass:AT1}--\ref{ass:AT3}, by Theorem~\ref{thm:semigp} there exists an evolution family $U^v : \Delta \times \Omega \to \mathcal{L}(X)$. In particular, \eqref{eq:dt-semigroup} holds.} Assume $x\in \mathcal{D}_{v}$, then integrating \eqref{eq:dt-semigroup} with respect to time leads to
\begin{equation}
    U^{v}(t,t)x-U^{v}(t,0)x=-\int_{0}^{t}U^{v}(t,s)A_{v}(s)x\,\textnormal{d}s.
\end{equation}
Testing the above identity by $x^\ast \in X^\ast$, we obtain 
\begin{equation}
\label{eq:ftc1}
       \int_{0}^{t}\left\langle U^{v}(t, s) A_{v}(s) x, x^{*}\right\rangle \textnormal{d}s = \left\langle U^{v}(t, 0) x, x^{*}\right\rangle - \left\langle x, x^{*}\right\rangle.
\end{equation}
Next we verify that \eqref{eq:ftc1} holds for $x^{*} \in \mathcal{D}_{v}^{*}$ and $x\in X$. This is justified by \eqref{eq:A-gUA1+b}, which implies that for every $s<t$, and fixed $v$, the operator $A_{v}^{-\lambda}(t)U^{v}(t,s)A_{v}^{1+\gamma}(s)$, for $\lambda > \delta^\ast$ and $\gamma < \delta^\ast$, can be uniquely extended to a bounded linear operator on $X$. The claim follows regarding that
\begin{align*}
    &\int_{0}^{t}\left\langle U^{v}(t, s) A_{v}(s) x, x^{*}\right\rangle\ds\\
    &=\int_{0}^{t}\left\langle (-A_{v}(t))^{-\lambda} U^{v}(t, s)(-1) (-A_{v}(s))^{1+\gamma}(-A_{v}(s))^{-\gamma} x, ((-A_{v}(t))^{\lambda})^{*} x^{*}\right\rangle\ds\\
    &\leq \int_{0}^{t} \frac{\|(-A_{v}(s))^{-\gamma} x\|_{X}}{(t-s)^{1+\gamma-\lambda}}\|  ((-A_{v}(t))^{\lambda})^{*} x^{*}\|_{X^{*}}\ds,
\end{align*}
 where we used \eqref{eq:A-gUA1+b} in the last step.
\end{proof}


\section{The main result}
\label{sec:main_thm}
In this section we introduce two solution concepts: pathwise mild and weak, for the quasilinear SPDE
\begin{equation}\label{eq:3.Qspde}
    \begin{cases}
    \textnormal{d}u = A_u u \,\dt + \sigma(u)\,\dW_t\\
    u(0)=u^0.
    \end{cases}
\end{equation}
We are interested in showing the equivalence of these solution concepts. As already mentioned, Pronk and Veraar introduced in \cite{PrVe15} the concept of pathwise mild solution for semilinear SPDEs with random nonautonomous generators, which was later extended in \cite{KuNe18} to quasilinear problems. The precise definition of the pathwise mild solution is given below.

\begin{definition}[Local pathwise mild solution]\label{def:loc_sol} \label{def:sol_glo}
A local pathwise mild solution for \eqref{eq:3.Qspde} is a pair $(u, \tau)$ where $\tau$ is a strictly positive stopping time and $\{u(t) : t \in [0,\tau)\}$ is a $\mathcal{Z}$-valued $\left(\mathcal{F}_{t}\right)_{t\in [0, \tau)}$-adapted stochastic process which satisfies for every $t \in [0,\tau)$, $\mathbb{P}$-a.s.
\begin{equation}
u(t) =U^{u}(t, 0) u_{0}+U^{u}(t, 0) \int_{0}^{t} \sigma(u(s)) \mathrm{d}W_s-\int_{0}^{t} U^{u}(t, s) A_{u}(s) \int_{s}^{t} \sigma(u(\tau)) \mathrm{d} W_\tau\, \mathrm{d} s,\label{mild}
\end{equation}
where $U^{u}$ is the evolution family generated by the operator $A_{u}$.
\end{definition}

\begin{definition}[Maximal solution] 
We call $\{u(t): t \in[0, \tau)\}$ a maximal local pathwise mild solution of \eqref{eq:3.Qspde} if for any other local pathwise mild solution $\{\widetilde{u}(t)\colon t \in[0, \widetilde{\tau})\}$ satisfying {$\tilde{\tau} \geq \tau$} a.s. and $\left.\tilde{u}\right|_{[0, \tau)}$ is equivalent to $u,$ one has $\tilde{\tau}=\tau$ a.s. If $\{u(t)\colon t \in[0, \tau)\}$ is a maximal local pathwise mild solution of \eqref{eq:3.Qspde}, then the stopping time $\tau$ is called its lifetime.
\end{definition}

In \cite{KuNe18}, the existence of a maximal local pathwise mild solution was established under additional regularity assumptions on the nonlinear term  $\sigma$. These {were} necessary for the fixed point argument. For the convenience of the reader, we recall these assumptions.

\begin{assumption}[Existence of pathwise mild solution] \label{ass:ex_path_mild}\mbox{}
\begin{itemize}
    \item[\textnormal{(1)}]  {\bf Assumption \ref{ass:AT}} holds.
    
    \item[\textnormal{(2)}] The mapping $\sigma: \Omega \times[0, T] \times X \rightarrow \mathcal{L}_{2}\left(H, \mathcal{D}^{2\beta}_u\right)$ is locally Lipschitz continuous, meaning that there exist constants $L_\sigma = L_\sigma(u,v) > 0$, $l_\sigma = l_\sigma(u) > 0$ such that
    \begin{align*}
    &\|\sigma(u)-\sigma(v)\|_{\mathcal{L}_{2}\left(H, \mathcal{D}^{2\beta}_u\right)} \leq L_{\sigma}\|u-v\|_{X},\qquad u,v\in Z,\\
    &\|\sigma(u)\|_{\mathcal{L}_{2}\left(H, \mathcal{D}^{2\beta}_u\right)} \leq l_{\sigma}\left(1+\|u\|_{X}\right),\qquad \qquad \;\, u\in Z.
    \end{align*}
\end{itemize}
\end{assumption}

 Under the above assumptions, the following  existence result for~\eqref{eq:3.Qspde} holds true~\cite[Theorem 3.11]{KuNe18}.

\begin{theorem}\label{thm:ex-path-mild}
Let {\bf Assumption~\ref{ass:ex_path_mild}} hold true. Then, there exists a unique maximal local pathwise mild solution $u$ for~\eqref{eq:3.Qspde} such that $u \in L^{0}\left(\Omega ; \mathcal{B}\left(\left[0, \tau_{\infty}\right) ; \mathcal{D}^{\beta}_u\right)\right) \cap$ $L^{0}\left(\Omega ; \mathcal{C}^{\delta}\left(\left[0, \tau_{\infty}\right) ; \mathcal{D}^{\alpha}_u\right)\right)$, where
 $\alpha\in[0,1)$, $\beta\in(1/2,1)$, $\beta>\alpha$ and $\delta\in(0,\beta-\alpha)$.
\end{theorem}

We now give the definition of a weak solution of \eqref{eq:3.Qspde}.

\begin{definition} \label{def:sol_weak}
We call an $(\mathcal{F}_t)_{t\geq 0}$-adapted $\mathcal{Z}$-valued process $u$ a weak solution of \eqref{eq:3.Qspde}, if the following identity holds $\mathbb{P}$-a.s.
\begin{align}
\label{weak}
\langle u(t), \varphi(t)\rangle & =\left\langle u^0,\varphi(0)\right\rangle - \int_{0}^{t}a\left(u(s); u(s), \varphi(s)\right) \ds+\int_{0}^{t}\left\langle u(s), \varphi^{\prime}(s)\right\rangle \ds \\ 
&\quad  -\int_{0}^{t}\left\langle \int_{0}^{s}\sigma(u(\tau))\dW_{\tau}, \varphi^{\prime}(s)\right\rangle~\ds+\left\langle \int_{0}^{t}\sigma(u(s))\dW_{s}, \varphi(t)\right\rangle,\nonumber 
\end{align}
for every $\varphi\in L^{0}\left(\Omega ; C^{1}\left([0, T]; \mathcal{D}_{u}^{*})\right)\right)$, where $a(u;\cdot,\cdot)\colon V\times V \to \R$ is the bilinear form corresponding to the operator $A_{u}$, as introduced in \eqref{eq:sesquilinear}. 
\end{definition}

\begin{remark}
 Note that the test functions $\varphi$ are allowed to depend on time and on $\omega$. We do not make any assumption regarding the adaptedness of these test functions, but only on their spatial and temporal regularity.
\end{remark}

\begin{remark}
\label{rem:ftc}
In comparison to the standard weak formulation of \eqref{eq:3.Qspde}, the weak formulation \eqref{weak} contains two additional terms 
\begin{align*}
    L_1 & := \int_{0}^{t}\left\langle u(s), \varphi^{\prime}(s)\right\rangle \ds, \\
    L_2 & := - \int_{0}^{t}\left\langle \int_{0}^{s}\sigma(u(\tau))\dW_{\tau}, \varphi^{\prime}(s)\right\rangle~\ds.
\end{align*}
The term $L_1$ appears quite naturally when we test \eqref{eq:3.Qspde} with a time-dependent function $\varphi$, apply the product rule, i.e.
\[\left\langle \frac{d}{ds}u(s),\varphi(s)\right\rangle = \frac{d}{ds} \left\langle u(s), \varphi(s)\right\rangle - \left\langle u(s), \varphi^\prime(s)\right\rangle\]
and integrate over time. The term $L_2$ can be again justified by the product rule, more precisely
$$
\left\langle \sigma(u(s))\dW_{s},\varphi(s)\right\rangle=\frac{d}{ds}\left\langle \int_{0}^{s}\sigma(u(r))\dW_{r},\varphi(s)\right\rangle-\left\langle \int_{0}^{s}\sigma(u(r))\dW_{r},\varphi^{'}(s)\right\rangle.$$
Integrating this relation with respect to the time variable, and interchanging the stochastic integral and the inner product gives the representation in Definition~\ref{def:sol_weak}.
\end{remark}

By Definition \ref{def:sol_weak}, the weak formulation of equation \eqref{eq:system} is given by
\begin{align}
\label{weak example}
\langle u(t), \varphi(t)\rangle &= \left\langle u^0,\varphi(0)\right\rangle - \int_{0}^{t}\left\langle \nabla u(s), B^{T}(u)\nabla\varphi(s) \right\rangle \ds+\int_{0}^{t}\left\langle u(s), \varphi^{\prime}(s)\right\rangle \ds \\ 
&\qquad-\int_{0}^{t}\left\langle \int_{0}^{s}\sigma(u(\tau))\dW_{\tau}, \varphi^{\prime}(s)\right\rangle~\ds + \left\langle \int_{0}^{t}\sigma(u(s))~\dW_{s}, \varphi(t)\right\rangle. \nonumber
\end{align}

\noindent \textbf{The key idea.} We describe the intuition behind the approach we use in order to prove the equivalence of pathwise mild and weak solutions for~\eqref{eq:3.Qspde}. Such an idea is standard in the context of quasilinear problems. Similar to the proof of existence of solutions~\cite{Ya10,KuNe18}, we firstly work with nonautonomous equations with random coefficients. More precisely, instead of treating~\eqref{eq:3.Qspde} we consider the linear equation
\begin{align}
    \begin{cases}\label{linear:Q}
    \textnormal{d}u = A_v u \,\dt + \sigma(u)\,\dW_t\\
    u(0)=u^0,
    \end{cases}
\end{align}
where $A_v u =\diver(B(v)\nabla u )$ as specified in~\eqref{eq:operator}. This is a linear nonautonomous Cauchy problem, which due to~\cite{PrVe15}, has a unique pathwise mild solution 
\begin{align}\label{mild:nonautonomous}
 u(t)=U^v(t,0)u^{0} - \int_{0}^{t} U^v(t, s) A_v(s) \int_{s}^{t}\sigma(u(\tau))\dW_{\tau}\, \ds + U^v(t, 0) \int_{0}^{t}\sigma(u(s))\dW_{s}.
\end{align}
 Its weak formulation reads as
\begin{align}
\langle u(t), \varphi(t)\rangle &= \left\langle u^0,\varphi(0)\right\rangle + \int_{0}^{t}\left\langle A_{v}u, \varphi(s) \right\rangle \ds+\int_{0}^{t}\left\langle u(s), \varphi^{\prime}(s)\right\rangle \ds\label{weakl} \\ 
&\quad-\int_{0}^{t}\left\langle \int_{0}^{s}\sigma(v(\tau))\dW_{\tau}, \varphi^{\prime}(s)\right\rangle~\ds + \left\langle \int_{0}^{t}\sigma(v(s))~\dW_{s}, \varphi(t)\right\rangle, \nonumber
\end{align}
for every test function $\varphi\in L^{0}\left(\Omega ; C^{1}\left([0, T]; \mathcal{D}_{u}^{*})\right)\right)$.
According to~\cite[Section 4.4]{PrVe15}, the weak solution~\eqref{weakl} is equivalent with the pathwise mild solution~\eqref{mild:nonautonomous}.
Returning to the quasilinear problem~\eqref{eq:3.Qspde}, 
one can show by means of fixed point arguments as in~\cite{KuNe18}, that under {\bf Assumption~\ref{ass:ex_path_mild} }
\begin{align*}
 u(t)=U^u(t,0)u^{0} - \int_{0}^{t} U^u(t, s) A_u(s) \int_{s}^{t}\sigma(u(\tau))\dW_{\tau}\, \ds + U^u(t, 0) \int_{0}^{t}\sigma(u(s))\dW_{s},
\end{align*}
is the pathwise mild solution of~\eqref{eq:3.Qspde}. {In conclusion, we approach the quasilinear SPDE \eqref{eq:3.Qspde} as a semilinear SPDE where the nonlinear map $A_u$ is viewed as a linear nonautonomous map for a fixed $u$, which is the pathwise mild solution of the quasilinear SPDE \eqref{eq:3.Qspde}. In this case, the tools  developed in~\cite[Section 4.4]{PrVe15} can be employed in order to establish the equivalence of weak~\eqref{weak} and pathwise mild~\eqref{mild} solutions in the quasilinear case.}\\

In order to prove our main result, we make the following assumptions.

\begin{assumption}
\mbox{}\label{assumptions}
\begin{enumerate}
    \item[\textnormal{(1)}] 
    For simplicity, we assume that the initial condition $u^{0}=0$.
    \item[\textnormal{(2)}] 
    {\bf Assumption~\ref{ass:AT},~\ref{ass:ATadjoint}} and {\bf~\ref{ass:const_dom}} are satisfied.
    \item[\textnormal{(3)}] 
The diffusion coefficient $\sigma : \Omega \times [0,T] \times X \to \mathcal{L}_{2}\left(H,X\right)$ is a  strongly measurable adapted process. Furthermore $\sigma(u)\in L^{0}(\Omega,L^{p}(0,T;\mathcal{L}_{2}(H,X)))$ for  $p\in(2,\infty)$.
\end{enumerate} 
\end{assumption}

Similar to~\cite{PrVe15} we introduce an appropriate space that incorporates the time and space regularity of the test functions.
\begin{definition} 
\label{def:gamma}
For $t \in[0, T]$ and $\beta\geq 0$, we let $\Gamma_{t,\beta}^{u}$ be the subspace of all test functions $\varphi \in L^{0}\left(\Omega ; C^{1}\left([0, t] ; X^{*}\right)\right)$, such that
\begin{enumerate}
    \item[\textnormal{(1)}] for all $s \in [0, t)$ and $\omega \in \Omega$, we have $\varphi(s) \in \mathcal{D}_{u}^{(\beta+1)*}$ and $\varphi^{\prime}(s) \in \mathcal{D}_{u}^{\beta *}$.
    
    \item[\textnormal{(2)}] the process $s \mapsto A_{u}^{*}(s) \varphi(s)$ belongs to $L^{0}\left(\Omega ; C\left([0, t] ; X^{*}\right)\right)$.
    
    \item[\textnormal{(3)}] There is a mapping $C : \Omega \rightarrow \mathbb{R}_{+}$ and  $\varepsilon > 0$ such that for all $s \in[0, t)$ 
    \begin{equation*}
    \qquad\left\|\left((-A_{u}(s))^{1+\beta}\right)^{*} \varphi(s)\right\|_{X^{*}}+\left\|\left((-A_{u}(s))^{\beta}\right)^{*} \varphi^{\prime}(s)\right\|_{X^{*}} \leq C(t-s)^{-1+\varepsilon}\label{inequality test function}.
\end{equation*}
\end{enumerate}
\end{definition}

In the following, we use test functions of the {form} $\varphi(s) = U(t,s)^{\ast} x^{\ast}$ for $x^\ast \in \mathcal{D}_{u}^{*}$. Thus, we need to show that such a $\varphi  \in \Gamma_{t, \beta}^{u}$ for some $\beta > 0$. This is established in the next {lemma}. We recall that the parameters $\delta$ respectively $\delta^\ast$ stand for the H\"older exponents in the (AT) conditions for $A_u$ and $A_u^{\ast}$ as specified in~\ref{ass:AT3} and~\ref{ass:AT3ad}{, respectively}.
\begin{lemma}
\label{lem.test}
Let $x^{*}\in \mathcal{D}_{u}^{*}$ and $t\in[0,T]$. For $\beta \in [0,\delta^{\ast})$, the process $\varphi:[0,t]\times\Omega\to X^{*}$ defined as $\varphi(s) := U^{u}(t,s)^{\ast}x^{\ast} $ belongs to $ \Gamma^{u}_{t,\beta}$.
\end{lemma}

\begin{proof}
In order to show that $\varphi(s) \in \Gamma_{t,\beta}^u$ we verify the three conditions of Definition~\ref{def:gamma}. For $(1)$, using the definition of the norm on $\mathcal{D}_{u}^{(\beta+1)*}$, we have by \eqref{eq:A-gUA1+b} and Remark \ref{rem:op_symm}
\begin{align*}
&\|\varphi(s)\|_{\mathcal{D}_{u}^{(\beta+1)*}}  = \|\left((-A_{u}(s))^{1+\beta}\right)^{*}(U^{u}(t,s))^{*}x^{*}\|_{X^{*}}\\
&\qquad = \|\left((-A_{u}(s))^{\beta+1}\right)^{*}(U^{u}(t,s))^{*}((-A_{u}(t))^{-\lambda})^{*}((-A_{u}(t))^{\lambda})^{*}x^{*}\|_{X^{*}}\\
    &\qquad \leq \|\left((-A_{u}(s))^{\beta+1}\right)^{*}(U^{u}(t,s))^{*}((-A_{u}(t))^{-\lambda})^{*}\|_{\mathcal{L}(X^{*})}\|((-A_{u}(t))^{\lambda})^{*}x^{*}\|_{X^{*}}\\
    &\qquad \leq \dfrac{C}{(t-s)^{1+\beta-\lambda}}\|((-A_{u}(t))^{\lambda})^{*}x^{*}\|_{X^{*}} < \infty,    
\end{align*}
where $\lambda\in (\beta,\delta^\ast)$.\\ 
Next we verify condition (2). Using \cite[Prop.~2.9]{AFT91}, we obtain
\[
\varphi^\prime(s) = \frac{d}{ds}\varphi(s) = -(A_{u}(s))^{*}(U^{u}(t,s))^{*}x^{*}.
\]
By Lemma \ref{lem:Semigroup estimates}, we immediately see that $\varphi$ is continuously differentiable for $s<t$. For the case $s = t$, we refer to \cite[Theorem~6.5]{AT92}. Now from the above identity, we have
\begin{align*}
    \|\varphi^\prime(s)\|_{\mathcal{D}_{u}^{\beta*}} & = \|\left((-A_{u}(s))^{\beta}\right)^{*}\varphi^\prime(s)\|_{X^{*}}\\
    & = \|\left((-A_{u}(s))^{1+\beta}\right)^{*}(U^{u}(t,s))^{*}x^{*}\|_{X^{*}} = \|\varphi(s)\|_{\mathcal{D}_{u}^{(\beta+1)*}} < \infty,
\end{align*}
using (1).
Condition $(3)$ of Definition~\ref{def:gamma} immediately follows from the previous two  estimates, choosing  $\varepsilon:=\lambda-\beta$.
\end{proof}

Due to {\bf Assumption~\ref{assumptions}} and regarding the regularity of the stochastic integral (recall Proposition ~\ref{prop:reg_stoc_int_1} an \ref{prop:reg_stoc_int_2}), the terms appearing in \eqref{weak} are well-defined for $\varphi(s)=U(t,s)^\ast x^\ast\in \Gamma^{u}_{t,\beta}$, where $x\in \mathcal{D}^{\ast}_u$.

\begin{lemma}
\label{lem:well-defined}
Let {\bf Assumption~\ref{assumptions}} hold and let $\beta\in[0,\delta^\ast)$. Then the following mappings   
\begin{align*}
    &u\mapsto \int_0^t a(u; u, \varphi(s))\,\ds,\\
    &u\mapsto \int_{0}^{t}\left\langle\int_{0}^{s}\sigma(u(\tau))\dW_{\tau},\varphi'(s)\right\rangle~\textnormal{d}s,\\
    &u\mapsto \int_{0}^{t}\left\langle U^u(t, s) A_{u}(s) \int_{s}^{t}\sigma(u(\tau))\dW_{\tau}, x^{*}\right\rangle\,\textnormal{d}s
\end{align*}
are well-defined from $\mathcal{Z}$ to $ L^{0}(\Omega;\R)$.
\end{lemma}

\begin{proof}
For the first term we obviously have
\begin{align*}
    &\left\vert a(u;u,\varphi(t))\right\vert=\left\vert\langle -A_{u}u,\varphi(t)\rangle\right\vert=\left\vert\langle u,(-A_{u})^{*}\varphi(t)\rangle\right\vert \leq C\|u\|_{X}\|(-A_{u})^{*}\varphi\|_{X}.
\end{align*}
For the second integral, we obtain
\begin{align*}
    &\int_{0}^{t}\left\vert\left\langle\int_{0}^{s}\sigma(u(\tau))\dW_{\tau},\varphi'(s)\right\rangle\right\vert~\textnormal{d}s = \int_{0}^{t}\left\vert\left\langle(-A_{u})^{-\beta}\int_{0}^{s}\sigma(u(\tau))\dW_{\tau},((-A_{u})^{\beta})^{*}\varphi'(s)\right\rangle\right\vert~\textnormal{d}s\\
    &\leq \int_{0}^{t}\left(\left\|(-A_{u})^{-\beta}\int_{0}^{s}\sigma(u(\tau))\dW_{\tau}\right\|_{X}\left\|((-A_{u})^{\beta})^{*}\varphi'(s)\right\|_{X^{*}}\right)\textnormal{d}s\\
    &\leq C\sup_{s\in[0,t]} \left\|(-A_{u})^{-\beta}\int_{0}^{s}\sigma(u(\tau))\dW_{\tau}\right\|_{X}\int_{0}^{t}\dfrac{1}{(t-s)^{1-\eps}}~\textnormal{d}s,
\end{align*}
where we use property (3) of Definition \ref{def:gamma}.

For the third term, setting $\varphi(s)~:=~U^{u}(t,s)^{\ast}x^{\ast}$ for $x^{*}\in \mathcal{D}_{u}^{*}$, we have
\begin{align*}
    &\int_{0}^{s}\left|\left\langle U^u(s, r) A_{u}(r) \int_{r}^{s}\sigma(u(\tau))\dW_{\tau}, x^{*}\right\rangle\right|\,\textnormal{d}r\\
    & =\int_{0}^{s}\left|\left\langle(-A_{u}(s))^{-\lambda} U^u(s, r)(A_{u}(r))(-A_{u}(r))^{\beta}(-A_{u}(r))^{-\beta} \int_{r}^{s}\sigma(u(\tau))\dW_{\tau},\left((-A_{u}(s))^{\lambda}\right)^{*} x^{*}\right\rangle\right|\,\textnormal{d}r\\
    &= \int_{0}^{s}\left|\left\langle (-A_{u}(r))^{-\beta}\int_{r}^{s}\sigma(u(\tau))\dW_{\tau},((-A_{u}(r))^{1+\beta})^{\ast}U^{u}(s, r)^\ast\left((-A_{u}(s))^{-\lambda}\right)^{\ast}\left((-A_{u}(s))^{\lambda}\right)^{*} x^{*}\right\rangle\right|\,\textnormal{d}r\\
    &\leq\int_{0}^{s} \bigg( \left\|(-A_{u}(r))^{-\beta}\int_{r}^{s}\sigma(u(\tau))\dW_{\tau}\right\|_{X} \\
    &\qquad \qquad \left\|((-A_{u}(r))^{1+\beta})^{*}U^{u}(s, r)^\ast\left((-A_{u}(s))^{-\lambda}\right)^{*}\left((-A_{u}(s))^{\lambda}\right)^{*} x^{*}\right\|_{X^{*}}\bigg)\,\textnormal{d}r \\
    &\leq C_{\sigma} \int_{0}^{s}(s-r)^{-1-\beta+\lambda}\left\|\left((-A_{u}(s))^{\lambda}\right)^{*} x^{*}\right\|_{X^{*}}\,\textnormal{d}r \leq C_{\sigma} s^{\lambda-\beta}\left\|\left((-A_{u}(s))^{\lambda}\right)^{*} x^{*}\right\|_{X^{*}}\\
    &\leq C^{\prime}_\sigma\left\|\left((-A_{u}(s))^{\lambda}\right)^{*} x^{*}\right\|_{X^{*}},
\end{align*}
 where we used \eqref{eq:A-gUA1+b} for $A_{u}^\ast(\cdot)$. The last term is bounded for $\lambda\in (\beta,\delta^\ast)$.
\end{proof}

Collecting all the previous deliberations, we now state the main result of this paper.

\begin{theorem}\label{thm:equivalence 1}
Let {\bf Assumption~\ref{assumptions}} be satisfied and let $\beta\in(0,\delta^\ast)$. Then the following assertions are valid.
\begin{itemize}
    \item[\textnormal{1.)}] If there exists a pathwise mild solution {$u$} for~\eqref{eq:3.Qspde} on the interval $[0,T]$. Then, {$u$} satisfies \eqref{weak} for all $\varphi\in\Gamma^{u}_{t,\beta}$, $t\in[0,T]$, $\mathbb{P}$-a.s.
    
    \item[\textnormal{2.)}] If there exists a weak solution $u$ for~\eqref{eq:3.Qspde} satisfying~\eqref{weak} for all $\varphi\in\Gamma^{u}_{t, \beta}$, {$t \in [0,T]$, $\mathbb{P}$-a.s.}. Then, $u$ satisfies \eqref{mild} for $t\in[0,T]$,  $\mathbb{P}$-a.s.
\end{itemize}
\end{theorem}

\begin{proof}
1.) We start by showing that a pathwise mild solution of \eqref{eq:3.Qspde} is also a weak solution. Assume that \eqref{mild} holds and fix $t\in[0,T]$. Let  $\lambda \in (\beta, \delta^\ast)$ and $u$ be the mild solution of~\eqref{eq:3.Qspde} with zero initial condition, i.e.
\begin{align}
\label{equiv1}
    &u(t)=-\int_{0}^{t}U^{u}(t,r)A_{u}(r)\int_{r}^{t}\sigma(u(\tau))\dW_{\tau}\dr + U^{u}(t,0)\int_{0}^{t}\sigma(u(\tau))\dW_{\tau}.
\end{align}
Applying $x^{\ast}\in \mathcal{D}_{u}^{\lambda *}$ to~\eqref{equiv1} and using that $\int\limits_{r}^t\sigma(u(\tau))\textnormal{d}W_\tau =\int\limits_{0}^t\sigma(u(\tau))\textnormal{d}W_\tau  -\int\limits_0^r \sigma(u(\tau))\textnormal{d}W_\tau $ to obtain
\begin{align*}
    \int_{0}^{t}\left\langle U^{u}(t, r) A_{u}(r) \int_{r}^{t}\sigma(u (\tau))\dW_{\tau}, x^{*}\right\rangle\dr
    & = \int_{0}^{t}\left\langle U^{u}(t, r) A_{u}(r) \int_{0}^{t}\sigma(u(\tau))\dW_{\tau}, x^{*}\right\rangle\dr \\
    & \quad - \int_{0}^{t}\left\langle U^{u}(t, r) A_{u}(r) \int_{0}^{r}\sigma(u(\tau))\dW_{\tau}, x^{*}\right\rangle\dr,
\end{align*}
further leads to
\begin{align}
\label{id2.1}
\langle u(t),x^{\ast}\rangle & = -\int_{0}^{t}\left\langle U^{u}(t,r)A_{u}(r)\int_{0}^{t}\sigma(u(\tau))\dW_{\tau},x^{\ast} \right\rangle\dr + \left\langle U^{u}(t,0)\int_{0}^{t}\sigma(u(\tau))\dW_{\tau},x^{\ast}\right\rangle \\
    &\quad +\int_{0}^{t}\left\langle U^{u}(t,r)A_{u}(r)\int_{0}^{r}\sigma(u(\tau))\dW_{\tau},x^{\ast} \right\rangle\dr. \nonumber
\end{align}
We use the identity \eqref{identity} for $x=\int_{0}^{t}\sigma(u(\tau))\dW_{\tau}$ and rewrite the first term on the right-hand-side as
\begin{align}
\label{equiv2}
     &\int_{0}^{t}\left\langle U^{u}(t, r) A_{u}(r) \int_{0}^{t}\sigma(u(\tau))\dW_{\tau}, x^{*}\right\rangle \dr \\
     &\; = \left\langle U^{u}(t, 0) \int_{0}^{t}\sigma(u(\tau))\dW_{\tau}, x^{*}\right\rangle - \left\langle \int_{0}^{t}\sigma(u(\tau))\dW_{\tau}, x^{*}\right\rangle.\nonumber
\end{align}
Using \eqref{equiv2} in \eqref{id2.1}, we obtain 
\begin{align}
\label{id2.2}
    \langle u(t),x^{\ast}\rangle & 
    = \int_{0}^{t}\left\langle U^{u}(t,r)A_{u}(r)\int_{0}^{r}\sigma(u(\tau))\dW_{\tau},x^{\ast} \right\rangle \dr + \left\langle\int_{0}^{t}\sigma(u(\tau))\dW_{\tau},x^{\ast}\right\rangle.
\end{align}
In order to show the equivalence with the weak solution, we need to use test functions $\varphi \in \Gamma^u_{t,\beta}$. To this aim we choose $x^{\ast} = \left(-A_{u}(s)\right)^{\ast}\varphi(s)$ {in \eqref{id2.2} and integrate over time to obtain} 
\begin{align}
    \label{id3.3}
    &\int_{0}^{t}\left\langle u(s), \left(-A_{u}(s)\right)^{*} \varphi(s)\right\rangle\ds - \int_{0}^{t}\left\langle \int_{0}^{s}\sigma(u(\tau))\dW_{\tau}, \left(-A_{u}(s)\right)^{*} \varphi(s)\right\rangle\ds\\
    &=\int_{0}^{t} \int_{0}^{s}\left\langle U^{u}(s, r) A_{u}(r) \int_{0}^{r}\sigma(u(\tau))\dW_{\tau} , \left(-A_{u}(s)\right)^{*} \varphi(s)\right\rangle\dr\,\ds\nonumber\\
    &=\int_{0}^{t} \int_{r}^{t}\left\langle U^{u}(s, r) A_{u}(r) \int_{0}^{r}\sigma(u(\tau))\dW_{\tau}, \left(-A_{u}(s)\right)^{*} \varphi(s)\right\rangle \ds\,\dr, \nonumber
\end{align}
where we used Fubini's theorem in the last step.
Using \eqref{eq:dt-semigroup}, we obtain for all $x \in X$ and $0 \leq r \leq t \leq T$,
\begin{align*}
    & \int_{r}^{t}\left\langle U^{u}(s, r) A_{u}(r) x, \left(-A_{u}(s)\right)^{*} \varphi(s)\right\rangle\ds = -  \int_{r}^{t}\left\langle A_{u}(s) U^{u}(s, r) A_{u}(r) x, \varphi(s)\right\rangle\ds\nonumber\\
    & = -\int_{r}^{t}\left\langle \dfrac{d}{\ds}U^{u}(s, r) A_{u}(r) x, \varphi(s)\right\rangle\ds\nonumber\\
    & = - \int_{r}^{t}\dfrac{d}{\ds}\left\langle U^{u}(s, r) A_{u}(r) x, \varphi(s)\right\rangle\ds + \int_{r}^{t}\left\langle U^{u}(s, r) A_{u}(r) x, \varphi^{\prime}(s)\right\rangle\ds\nonumber\\
    & = - \left\langle U^{u}(t, r) A_{u}(r) x, \varphi(t)\right\rangle + \left\langle U^{u}(r, r) A_{u}(r) x, \varphi(r)\right\rangle + \int_{r}^{t}\left\langle U^{u}(s, r) A_{u}(r) x, \varphi^{\prime}(s)\right\rangle \ds.\nonumber
\end{align*}
Furthermore
\begin{align} 
\label{second identity}
    & \langle U^{u}(t, r) A_{u}(r) x, \varphi(t)\rangle + \left\langle x, \left(-A_{u}(r)\right)^{*} \varphi(r)\right\rangle \\
    &\quad = - \int_{r}^{t}\left\langle U^{u}(s, r) A_{u}(r) x, \left(-A_{u}(s)\right)^{*} \varphi(s)\right\rangle\ds + \int_{r}^{t}\left\langle U^{u}(s, r) A_{u}(r) x, \varphi^{\prime}(s)\right\rangle\ds. \nonumber
\end{align}
Note that the expressions above are  well-defined. Indeed, for $\varphi \in \Gamma_{t, \beta}^u$ and $x\in X$, choosing $1>\lambda>\theta >0$ and applying\eqref{eq:A-gUA1+b}, we infer that
\begin{align*}
    &\left|\left\langle U^{u}(s, r) A_{u}(r) x, \left(-A_{u}(s)\right)^{*} \varphi(s)\right\rangle\right|  = \left|\left\langle(-A_{u}(s))^{-\lambda} U^{u}(s, r) A_{u}(r) x,\left((-A_{u}(s))^{1+\lambda}\right)^{*} \varphi(s)\right\rangle\right|\\
&\qquad \leq \|(-A_{u}(s))^{-\lambda } U^{u}(s, r) (-A_{u}(r))^{1+\theta}(-A_{u}(r))^{-\theta} x\|_X\|\left((-A_{u}(s))^{1+\lambda}\right)^{*} \varphi(s)\|_{X^\ast} \\
& \qquad \le C(s-r)^{-1+\lambda-\theta}(t-s)^{-1+\varepsilon}\|(-A_{u}(r))^{-\theta} x\|_X\\
&\qquad \le C(s-r)^{-1+\lambda-\theta}(t-s)^{-1+\varepsilon}\|x\|_X.
\end{align*}
Setting $x = \int_0^r \sigma(u(\tau))\dW_\tau$ in \eqref{second identity} further leads to
\begin{align}
\label{equiv3}
    & \left\langle U^{u}(t, r) A_{u}(r) \int_{0}^{r}\sigma(u(\tau))\dW_{\tau}, \varphi(t)\right\rangle + \left\langle \int_{0}^{r}\sigma(u(\tau))\dW_{\tau}, \left(-A_{u}(r)\right)^{*} \varphi(r)\right\rangle\\
    & = - \int_{r}^{t}\left\langle U^{u}(s, r) A_{u}(r) \int_{0}^{r}\sigma(u(\tau))\dW_{\tau}, \left(-A_{u}(s)\right)^{*} \varphi(s)\right\rangle\ds \nonumber \\
    &\qquad + \int_{r}^{t}\left\langle U^{u}(s, r) A_{u}(r) \int_{0}^{r}\sigma(u(\tau))\dW_{\tau}, \varphi^{\prime}(s)\right\rangle\ds. \nonumber
\end{align}
Next we use \eqref{equiv3} to deal with the right-hand-side of \eqref{id3.3}. From \eqref{id3.3} we get
\begin{align}
\label{equiv4}
     \int_{0}^{t}\left\langle u(s), \left(-A_{u}(s)\right)^{*} \varphi(s)\right\rangle\ds & =\int_{0}^{t} \int_{r}^{t}\left\langle U^{u}(s, r) A_{u}(r) \int_{0}^{r}\sigma(u(\tau))\dW_{\tau}, \left(-A_{u}(s)\right)^{*} \varphi(s)\right\rangle \ds\,\dr \\
     &\quad + \int_{0}^{t}\left\langle \int_{0}^{s}\sigma(u(\tau))\dW_{\tau}, \left(-A_{u}(s)\right)^{*} \varphi(s)\right\rangle\ds. \nonumber
\end{align}
Thus, using \eqref{equiv3} in \eqref{equiv4} results in
\begin{align*}
    \int_{0}^{t}\left\langle u(s), \left(-A_{u}(s)\right)^{*} \varphi(s)\right\rangle\ds & = - \int_{0}^{t}\left\langle U^{u}(t, r) A_{u}(r) \int_{0}^{r}\sigma(u(\tau))\dW_{\tau}, \varphi(t)\right\rangle\dr \\
    & \quad -\int_{0}^{t}\left\langle \int_{0}^{r}\sigma(u(\tau))\dW_{\tau}, \left(-A_{u}(r)\right)^{*} \varphi(r)\right\rangle\dr \\
    & \quad + \int_{0}^{t}\int_{r}^{t}\left\langle U^{u}(s, r) A_{u}(r) \int_{0}^{r}\sigma(u(\tau))\dW_{\tau}, \varphi^{\prime}(s)\right\rangle\ds\,\dr \\ 
    & \quad + \int_{0}^{t}\left\langle \int_{0}^{s}\sigma(u(\tau))\dW_{\tau}, \left(-A_{u}(s)\right)^{*} \varphi(s)\right\rangle\ds.
\end{align*}
The above expression on simplification results in
\begin{align}
    \label{eq:3identity}
    \int_{0}^{t}\left\langle u(s), \left(-A_{u}(s)\right)^{*} \varphi(s)\right\rangle\ds & = - \int_{0}^{t}\left\langle U^{u}(t, r) A_{u}(r) \int_{0}^{r}\sigma(u(\tau))\dW_{\tau}, \varphi(t)\right\rangle\dr \\ 
    &\quad + \int_{0}^{t}\int_{r}^{t}\left\langle U^{u}(s, r) A_{u}(r) \int_{0}^{r}\sigma(u(\tau))\dW_{\tau}, \varphi^{\prime}(s)\right\rangle\ds\,\dr. \nonumber
\end{align}
Moreover, choosing $x^{\ast} = \varphi(t)$ in \eqref{id2.2}, we get
\begin{align}
\label{eq:mildtest}
    \left\langle u(t), \varphi(t)\right\rangle & = \int_{0}^{t}\left\langle U^{u}(t, r) A_{u}(r)\int_{0}^{r}\sigma(u(\tau))\dW_{\tau}, \varphi(t)\right\rangle\dr
    \\
    &\quad + \left\langle \int_{0}^{t}\sigma(u(\tau))\dW_{\tau}, \varphi(t)\right\rangle. \nonumber
\end{align}
Using \eqref{eq:mildtest} and \eqref{eq:3identity}, we obtain
\begin{align}
   \label{equiv5}
   \int_{0}^{t}\left\langle u(s),\left(-A_{u}(s)\right)^\ast \varphi(s)\right\rangle\ds & = - \left\langle u(t),\varphi(t)\right\rangle + \left\langle\int_{0}^{t}\sigma(u(\tau))\dW_{\tau},\varphi(t)\right\rangle\\
   &\quad + \int_{0}^{t}\int_{r}^{t}\left\langle U^{u}(s, r) A_{u}(r) \int_{0}^{r}\sigma(u(\tau))\dW_{\tau}, \varphi^{\prime}(s)\right\rangle\ds\,\dr. \nonumber
 \end{align}
 Further, choosing $x^\ast = \varphi^\prime(t)$ in \eqref{id2.2}, we can express the integrand of the last term in the right-hand-side of \eqref{equiv5} as
 \begin{align}
\label{equiv6}
\int_{0}^{t}\left\langle U^{u}(t, r) A_{u}(r)\int_{0}^{r}\sigma(u(\tau))\dW_{\tau}, \varphi^\prime(t)\right\rangle\dr = 
    \left\langle u(t), \varphi^\prime(t)\right\rangle -
\left\langle \int_{0}^{t}\sigma(u(\tau))\dW_{\tau}, \varphi^\prime(t)\right\rangle. 
\end{align}
Using Fubini's theorem and plugging in the relation \eqref{equiv6} in \eqref{equiv5}, we infer that
  \begin{align*}
   \int_{0}^{t}\left\langle u(s),\left(-A_{u}(s)\right)^\ast \varphi(s)\right\rangle\ds & = - \left\langle u(t),\varphi(t)\right\rangle + \left\langle\int_{0}^{t}\sigma(u(\tau))\dW_{\tau},\varphi(t)\right\rangle\\
   & + \int_{0}^{t}\left\langle u(s),\varphi^{\prime}(s)\right\rangle\ds 
   - \int_{0}^{t}\left\langle\int_{0}^{s}\sigma(u(\tau))\dW_{\tau},\varphi^{\prime}(s)\right\rangle\ds.
\end{align*}
From the previous expression we conclude that $u$ satisfies \eqref{weak} and is therefore a weak solution of~\eqref{eq:3.Qspde}.\\

\noindent 2.) Let $u$ be a weak solution of \eqref{eq:3.Qspde}. Then $u \in \mathcal{Z}$ and satisfies the weak formulation~\eqref{weak}, namely
\begin{align*}
 \langle u(t), \varphi(t)\rangle &= - \int_{0}^{t}a\left( u(s); u(s), \varphi(s) \right) \ds+\int_{0}^{t}\left\langle u(s), \varphi^{\prime}(s)\right\rangle \ds \nonumber \\ 
& \quad-\int_{0}^{t}\left\langle \int_{0}^{s}\sigma(u(\tau))~\dW_{\tau}, \varphi^{\prime}(s)\right\rangle\ds+\left\langle \int_{0}^{t}\sigma(u(\tau))\dW_{\tau}, \varphi(t)\right\rangle \nonumber.
\end{align*}
In particular, by Theorem~\ref{thm:semigp} there exists an evolution family $U^u \colon \Delta \times \Omega \to \mathcal{L}(X)$ generated by the operator $A_u$. Using the relation between $a(u;\cdot,\cdot)$ and $A_u$, recall~\eqref{eq:sesquilinear},  
the previous expression rewrites as 
\begin{align}
\label{equiv7}
\langle u(t), \varphi(t)\rangle &= - \int_{0}^{t}\left\langle u(s), \left(-A_{u}(s)\right)^{\ast}\varphi(s) \right\rangle \ds+\int_{0}^{t}\left\langle u(s), \varphi^{\prime}(s)\right\rangle \ds \\ 
&\quad-\int_{0}^{t}\left\langle \int_{0}^{s}\sigma(u(\tau))~\dW_{\tau}, \varphi^{\prime}(s)\right\rangle\ds+\left\langle \int_{0}^{t}\sigma(u(\tau))\dW_{\tau}, \varphi(t)\right\rangle. \nonumber
\end{align}
From Lemma~\ref{lem.test}, for a fixed $t\in[0,T]$ and $x^{\ast}\in \mathcal{D}_{u}^{*}$, $\varphi(s)=U^{u}(t, s)^{*} x^{*}$ belongs to $\Gamma^{u}_{t,\beta}$. Now with this choice of test functions and using that $\varphi^{\prime}(s) = \left(-A_{u}(s)\right)^{*} \varphi(s)$, we obtain from~\eqref{equiv7} that
\begin{align*}
&\langle u(t), U^{u}(t, t)^{*} x^{*}\rangle =  - \int_{0}^{t}\left\langle u(s), \left(-A_{u}(s)\right)^{\ast}U^{u}(t, s)^{*} x^{*} \right\rangle \ds \\
&\qquad + \int_{0}^{t}\left\langle u(s), (-A_{u}(s))^{\ast}\varphi(s)\right\rangle \ds
- \int_{0}^{t}\left\langle \int_{0}^{s}\sigma(u(\tau))~\dW_{\tau}, \left(-A_{u}(s)\right)^{\ast}\varphi(s)\right\rangle\ds \\
&\qquad + \left\langle \int_{0}^{t}\sigma(u(\tau))~\dW_{\tau}, U^{u}(t, t)^{*} x^{*}\right\rangle,
\end{align*}
which simplifies further to
\begin{align*}
\langle u(t), x^\ast\rangle & =  - \int_{0}^{t}\left\langle u(s), \left(-A_{u}(s)\right)^{\ast}U^{u}(t, s)^{*} x^{*} \right\rangle \ds \\
&\quad + \int_{0}^{t}\left\langle u(s), (-A_{u}(s))^{\ast}U^{u}(t, s)^{*} x^{*}\right\rangle \ds  + \left\langle \int_{0}^{t}\sigma(u(\tau))~\dW_{\tau}, x^{*}\right\rangle \\
& \quad
- \int_{0}^{t}\left\langle \int_{0}^{s}\sigma(u(\tau))\dW_{\tau}, \left(-A_{u}(s)\right)^{\ast}U^{u}(t, s)^{*} x^{*}\right\rangle\ds \\
& =   \left\langle \int_{0}^{t}\sigma(u(\tau))\dW_{\tau}, x^{*}\right\rangle -\int_{0}^{t}\left\langle U^{u}(t, s)\left(-A_{u}(s)\right)\int_{0}^{s}\sigma(u(\tau))\dW_{\tau},  x^{*}\right\rangle\ds.
\end{align*}
All in all, we obtained that
\begin{align*}
    &\left\langle u(t), x^{*}\right\rangle=\int_{0}^{t}\left\langle U^{u}(t, s) A_{u}(s) \int_{0}^{s}\sigma(u(\tau))\dW_{\tau}, x^{*}\right\rangle\ds + \left\langle \int_{0}^{t}\sigma(u(\tau))\dW_{\tau}, x^{*}\right\rangle.
\end{align*}
Splitting the stochastic integral in the first term on the right-hand-side, into two stochastic integrals results in
\begin{align}
\label{equiv8}
\langle u(t), x^\ast \rangle    & = \int_{0}^{t}\left\langle U^{u}(t, s) A_{u}(s) \left(\int_{0}^{t}\sigma(u(\tau))\dW_{\tau} - \int_{s}^{t}\sigma(u(\tau))\dW_{\tau}\right), x^{*}\right\rangle\ds \nonumber \\
&\quad + \left\langle \int_{0}^{t}\sigma(u(\tau))\dW_{\tau}, x^{*}\right\rangle \nonumber\\
    &=\int_{0}^{t}\left\langle U^{u}(t, s) A_{u}(s) \int_{0}^{t}\sigma(u(\tau))\dW_{\tau}, x^{*}\right\rangle\ds  +\left\langle \int_{0}^{t}\sigma(u(\tau))\dW_{\tau}, x^{*}\right\rangle \\
    &\quad - \int_{0}^{t}\left\langle U^{u}(t, s) A_{u}(s) \int_{s}^{t}\sigma(u(\tau))\dW_{\tau}, x^{*}\right\rangle\ds. \nonumber
\end{align}
Using the identity \eqref{identity} with $x = \int_0^t \sigma(u(\tau))\dW_\tau$, the first term on the right-hand-side of \eqref{equiv8} can be written as 
\begin{align*}
    &\int_{0}^{t}\left\langle U^{u}(t, s) A_{u}(s) \int_{0}^{t}\sigma(u(\tau))\dW_{\tau}, x^{*}\right\rangle\ds 
     \\
     &\quad 
    = \left\langle U^{u}(t,0)\int_{0}^{t}\sigma(u(\tau))\dW_{\tau}, x^{*}\right\rangle - \left\langle\int_{0}^{t}\sigma(u(\tau))\dW_{\tau}, x^{*}\right\rangle.
\end{align*}
Using the above identity in \eqref{equiv8} entails
\begin{align*}
\left\langle u(t), x^{*}\right\rangle = \left\langle U^{u}(t, 0) \int_{0}^{t}\sigma(u(\tau))\dW_{\tau}, x^{*}\right\rangle - \int_{0}^{t}\left\langle U^{u}(t, s) A_{u}(s) \int_{s}^{t}\sigma(u(\tau))\dW_{\tau}, x^{*}\right\rangle\ds.
\end{align*}
In the previous deliberations $x^\ast \in \mathcal{D}_{u}^{*}$. Since $\mathcal{D}_{u}^{*}$ is dense in $X$, the result can be extended to every $x^\ast \in X$ by the Hahn-Banach theorem. This justifies that the weak solution of~\eqref{eq:3.Qspde} satisfies the mild formulation \eqref{mild}.
\end{proof}

For time-independent test functions, we recover the equivalence of the pathwise mild solution with the following standard weak formulation  
\begin{align}
&\langle u(t), x^{\ast}\rangle= - \int_{0}^{t}a\left(u(s); u(s), x^{\ast}\right) \ds+\left\langle \int_{0}^{t}\sigma(u(s))\dW_{s}, x^{*}\right\rangle,\label{weak2} 
\end{align}
for $x^{\ast}\in \mathcal{D}_{u}^{*}$.

\begin{theorem} \label{thm:equivalence 2} 
Let {\bf Assumption~\ref{assumptions}} be satisfied. Then the following assertions are valid.
\begin{itemize}
\item[\textnormal{1.)}] If there exists a pathwise mild solution {$u$} for~\eqref{eq:3.Qspde} on the interval $[0,T]$. Then, {$u$} satisfies \eqref{weak2} for all $x^\ast \in \mathcal{D}_{u}^{*}$, $t \in [0,T]$, $\mathbb{P}$-a.s.
\item[\textnormal{2.)}]
If there exists a weak solution $u$ for~\eqref{eq:3.Qspde} satisfying~\eqref{weak2} for all $x^{\ast}\in \mathcal{D}_{u}^{*}$, $t \in [0,T]$, $\mathbb{P}$-a.s. Then, $u$ satisfies \eqref{mild} for $t\in[0,T]$,  $\mathbb{P}$-a.s.
\end{itemize}
\end{theorem}

\begin{proof}
1.) Let $t \in [0,T]$. We consider only the case $\sigma(u(t)) \in \mathcal{D}_{u}$ such that the process $t \mapsto A_{u}(t)\sigma(u(t))$ is adapted and belongs to $L^{0}(\Omega; L^{p}(0,T;\mathcal{L}_2(H,X)))$. 

For $x^{\ast}\in \mathcal{D}^{\ast}_{u}$, we simply set $\varphi(t):=x^{\ast}$. However, such a test function does not belong to $\Gamma_{t,\beta}^{u}$, since $\varphi \notin \mathcal{D}_{u}^{1+\beta}$. {Though}, the additional spatial regularity of $\sigma$ enables us to perform the same proof as before. From \eqref{id3.3},  we obtain  for the test function $x^\ast$ instead of $A_u(s)^\ast \varphi(s)$
\begin{align*}
    &\int_{0}^{t}\left\langle u(s),x^{*} \right\rangle\ds - \int_{0}^{t}\left\langle \int_{0}^{s}\sigma(u(\tau))\dW_{\tau}, x^{*}\right\rangle\ds 
    \\
    & \qquad 
    = \int_{0}^{t} \int_{r}^{t}\left\langle U^{u}(s, r) A_{u}(r) \int_{0}^{r}\sigma(u(\tau))\dW_{\tau}, x^{*}\right\rangle\ds\,\dr.
\end{align*}
Testing with $(-A_{u}(s))^{*}x^{*}$ instead of $(-A_{u}(s))^{*}\varphi(s)$ entails 
\begin{align*}
      \int_{r}^{t}\left\langle U^{u}(s, r) A_{u}(r) x, A_{u}(s)^{*} x^{*}\right\rangle ~\ds & = \int_{r}^{t}\dfrac{d}{ds}\left\langle U^{u}(s, r) A_{u}(r) x, x^{*}\right\rangle ~\ds\\
    & = \left\langle U^{u}(t, r) A_{u}(r) x, x^{*}\right\rangle-\left\langle A_{u}(r) x, x^{*}\right\rangle.
\end{align*} 
Analogously to the proof of Theorem \ref{thm:equivalence 1}, the result now follows for $\sigma(\cdot)\in \mathcal{D}_u$. The general case follows from a suitable approximation argument for $\sigma(\cdot)$ as in~\cite[Theorem 4.9]{PrVe15}.\\

\noindent 2.) We show now that a solution satisfying \eqref{weak2},  also verifies  \eqref{mild}. To this aim, we fix  $t \in [0,T]$, and let $f \in C^{1}([0,t])$, $x^{*}\in \mathcal{D}_{u}^{*}$. Using the density of $\mathcal{D}_u^{\ast}$ in $X$, it suffices to consider test functions of the form $\varphi(t) = f(t) \otimes x^{*}$.
Using such test functions in \eqref{weak2}, we infer that
\begin{align}
\label{equiv9}
\langle u(t), \varphi(t)\rangle - \left\langle \int_{0}^{t}\sigma(u(s))\dW_s, \varphi(t)\right\rangle = - \int_{0}^{t}\left\langle u(s), (-A_{u}(s))^{*} x^{*}\right\rangle\ds f(t). 
\end{align}
{The} integration by parts formula {results in}
\begin{align*}
& \int_{0}^{t} \int_{0}^{s}\left\langle u(r), \left(-A_{u}(r)\right)^{*} x^{*}\right\rangle \dr f^{\prime}(s)\ds \\
&\qquad =  \int_{0}^{t}\left\langle u(s), (-A_{u}(s))^{*} x^{*}\right\rangle \ds f(t)  - \int_{0}^{t}\left\langle u(s), \left(-A_{u}(s)\right)^{*} \varphi(s)\right\rangle \ds.
\end{align*}
 The above identity and \eqref{weak2} further entail that
\begin{align*}
&\langle u(t), \varphi(t)\rangle - \left\langle \int_{0}^{t}\sigma(u(s))\dW_s, \varphi(t)\right\rangle \\
&\; = - \int_{0}^{t}\left\langle u(s), \left(-A_{u}(s)\right)^{*} \varphi(s)\right\rangle \ds + \int_{0}^{t}\left\langle u(s), x^{*}\right\rangle f^{\prime}(s) \ds
\\
& \qquad 
- \int_{0}^{t}\left\langle \int_{0}^{s}\sigma(u(\tau))\dW_{\tau}, x^{*}\right\rangle f^{\prime}(s) \ds \\
&= - \int_{0}^{t}\left\langle u(s), \left(-A_{u}(s)\right)^{*} \varphi(s)\right\rangle \ds + \int_{0}^{t}\left\langle u(s), \varphi^{\prime}(s)\right\rangle\ds  -\int_{0}^{t}\left\langle \int_{0}^{s}\sigma(u(\tau))\dW_{\tau}, \varphi^{\prime}(s)\right\rangle\ds.
\end{align*}
Recalling $a(u;u,\varphi) = -\langle A_{u} u,\varphi\rangle$, we obtain from the previous expression the weak formulation \eqref{weak}, for test functions having the structure $\varphi(t) = f(t) \otimes x^{*}$.
To extend the previous identity to test functions belonging to $\Gamma_{\beta,t}^{u}$, we firstly extend it to simple functions $\varphi\colon \Omega \to C^{1}([0,t]; X^{*}) \cap C([0,t]; \mathcal{D}_{u}^{*})$. By an approximation argument, this can be further extended to {any function} $\varphi \in L^{0}(\Omega; C^{1}([0,t];X)\cap C([0,t];\mathcal{D}_{u}^{*}))$. Choosing an arbitrary $x^{*}\in \mathcal{D}_{u}^{*}$ and setting $\varphi(s) := U^{u}(t,s)^{*}x^{*} $, by Lemma~\ref{lem.test}, we have that $\varphi(s) \in L^{0}\left(\Omega; C^{1}([0,t]; X) \cap C([0,t]; \mathcal{D}_u^\ast)\right)$. This proves the statement
\end{proof}


\section{Examples}
\label{sec:example}
In this section, we present two examples of parabolic quasilinear SPDEs, to which the theory developed in this paper applies. 
The existence theory for pathwise mild-, martingale- and weak solutions for these problems is well-known, see \cite{KuNe18, DJZ19, HoZh17}. After introducing these SPDEs and recalling the corresponding existence results, we show that they satisfy our assumptions. Therefore, the pathwise mild and weak solution concepts are equivalent in {these cases}.

\begin{example}\textnormal{\textbf{(The stochastic SKT population model)}}
Let $\dom\subset\R^{2}$ be an open bounded domain with $\mathcal{C}^{2}$ boundary.  We fix parameters $\alpha_1, \alpha_2, \delta_{11}, \delta_{21} > 0$. We are interested in studying a cross-diffusion SPDE, which was originally introduced by Shigesada, Kawasaki and Teramato \cite{SKT79} in the deterministic setting, to analyze population segregation by induced cross-diffusion in a two-species model. Note that the nonlinear {drift} term correspond to those arising in the classical Lokta-Volterra competition model. The stochastic SKT system is given by 
\begin{equation}
\label{eq:SKT}
\begin{split}
&\mathrm{d} u_{1}=\left(\Delta\left(\alpha_{1} u_{1}+\gamma_{1} u_{1} u_{2}+\beta_{1} u_{1}^{2}\right)+\delta_{11} u_{1}-\theta_{11} u_{1}^{2}-\theta_{12} u_{1} u_{2}\right) \mathrm{d} t+\sigma_{1}(u_{1}, u_{2}) \mathrm{d} W^{1}_t, \\
&\mathrm{d} u_{2}=\left(\Delta\left(\alpha_{2} u_{2}+\gamma_{2} u_{1} u_{2}+\beta_{2} u_{2}^{2}\right)+\delta_{21} u_{2}-\theta_{21} u_{1} u_{2}-\theta_{22} u_{2}^{2}\right) \mathrm{d} t+\sigma_{2}(u_{1}, u_{2}) \mathrm{d} W^{2}_t,
\end{split}
\end{equation}
for $t \in[0,T]$ and $x \in \dom$ and is supplemented with the following boundary and initial conditions:
\begin{align*}
&\frac{\partial}{\partial n} u_{1}(t,x) = \frac{\partial}{\partial n} u_{2}(t,x) = 0, & t > 0, \; x \in \partial \dom,\\
&u_{1}(x, 0) = u_{1}^0(x) \geq 0,\quad  u_{2}(x, 0) = u_{2}^0(x) \geq 0, &x \in \dom.
\end{align*}
$W = (W^1, W^2)$ is an $H$-valued cylindrical Wiener process. The solution $u:=(u_1,u_2)$, where $u_1=u_1(x, t)$ and $u_2 = u_2(x, t)$
denote the densities of two competing species $S_1$ and $S_2$ at certain location $x \in \dom$, at time $t$.
The coefficients $\theta_{11}, \theta_{22} > 0$ denote the intraspecies competition rates in $S_1$, respectively in $S_2$
and $\theta_{12}, \theta_{21} > 0$ stand for the interspecies competition rates between $S_1$ and $S_2$. Furthermore, the terms $\Delta(\beta_1 u_1^2)$ and $\Delta(\beta_2 u_2^2)$ represent the self-diffusions of $S_1$ and $S_2$ with rates $\beta_1, \beta_2 \ge 0$, and $\Delta(\gamma_1 u_1 u_2)$, $\Delta(\gamma_2 u_1 u_2)$ represent the cross-diffusions of $S_1$ and $S_2$ with rates $\gamma_1, \gamma_2 \ge 0$.
\end{example}
The SKT system \eqref{eq:SKT} can be rewritten as an abstract quasilinear SPDE:
\begin{equation}
\label{eq:SKT_abs}
\begin{cases}
&\mathrm{d} u=[A_{u} u+F(u)]~ \mathrm{d} t+\sigma(u)~ \mathrm{d} W_t, \quad t \in[0, T] \\
&u(0)=u^{0},
\end{cases}
\end{equation}
where 
\[
A_{u}u:=\textnormal{div}(B(u)\nabla u)-\Gamma u,
\]
with 
\[B(u)=\begin{pmatrix}
\alpha_{1}+2\beta_1u_{1} + \gamma_1 u_{2} & \gamma_1 u_{1} \\
\gamma_2 u_{2} & \alpha_{2}+2 \beta_2 u_{2} +\gamma_2 u_{1}
\end{pmatrix}, \qquad 
\Gamma(u)=\begin{pmatrix}
\delta_{11} & 0 \\
0 & \delta_{21}
\end{pmatrix}.\]
The nonlinear term $F$ corresponds to the Lotka-Volterra type competition model
\[
F(u)=\begin{pmatrix}
2\delta_{11} u_{1}-\theta_{11} u_{1}^{2}-\theta_{12} u_{1} u_{2} \\
 2\delta_{21} u_{2}-\theta_{21} u_{1} u_{2}-\theta_{22} u_{2}^{2}
\end{pmatrix}.\]
In order to ensure the positive definiteness of the matrix $B$, the following restriction on the parameters is necessary:
    \begin{align}\label{pos:def:skt}
\gamma_{1}^{2}<8\alpha_{1}\beta_{1}\quad\mbox{and}\quad\gamma_{2}^{2}<8\alpha_{2}\beta_{2}.
\end{align}
This assumption is required in order to show that $A_u$ generates a parabolic evolution system $U^u$ for $u\in Z$, for a natural choice of $Z$, see below.
The (AT) conditions \ref{ass:AT1}--\ref{ass:AT3} are  satisfied for a standard choice of Hilbert spaces $Z\hookrightarrow Y \hookrightarrow X$, where $X:=L^2(\mathcal{O})$, 
 $Z:=H^{1+\varepsilon}(\mathcal{O})$ and $Y:=H^{1+\varepsilon_0}(\mathcal{O})$ for $0<\varepsilon_0<\varepsilon$, see \cite[Chapter~15.2.2]{Ya10}. Moreover, according to \cite[Chapter~1.8.2]{Ya10}, \ref{ass:AT1ad}--\ref{ass:AT3ad} are also satisfied by the adjoint operator $(A_u)^{\ast}$.
 \begin{remark}
\label{rem:space_cho}
The choice of the space $Z=H^{1+\varepsilon}(\dom)$ is natural. Using the Sobolev embedding  $W^{k, p}(\dom) \hookrightarrow C(\bar{\dom})$ for $kp > d$, we observe that for fixed $t\in[0,T]$ and $d=2$, the choice of $Z = H^{1+\eps}(\dom)$ ensures that \eqref{blf:continuity} is satisfied.
\end{remark}
 We also emphasize that the domains of the fractional powers of $A_u$, for $u\in Z$, can be identified with Sobolev spaces, see~\cite[Prop.~ 15.3]{Ya10}.  More precisely, we have
\begin{align*}
   \begin{cases}
\mathcal{D}^{\theta}_u = H^{2\theta}(\dom),~ \mbox{ for } 0\leq\theta<\frac{3}{4}\\
\mathcal{D}^{\theta}_u= H^{2\theta}_N(\dom),~ \mbox{ for } \frac{3}{4}<\theta\leq 1,
\end{cases} 
\end{align*}
where $H^{2\theta}_N(\dom)$ incorporates the Neumann boundary conditions.
 Furthermore, the nonlinear drift term is locally Lipschitz continuous on $X$. Letting $u^0\in Z$ a.s. and assuming a local Lipschitz continuity on $\sigma$ (recall {\bf Assumption~\ref{ass:ex_path_mild}}),~\cite[Theorem ~4.3]{KuNe18} provides the existence of a local-in-time pathwise mild solution $u$ of \eqref{eq:SKT_abs} such that $u \in L^{0}\left(\Omega ; \mathcal{B}\left(\left[0, \tau\right) ; Z\right)\right) \cap$ $L^{0}\left(\Omega ; \mathcal{C}^{\delta}\left(\left[0, \tau\right) ; Y\right)\right)$, where $\delta\in(0,\frac{\varepsilon-\varepsilon_0}{2})$.

The main result establishes the equivalence of this pathwise mild solution with the weak solution. This statement is a direct consequence of Theorem \ref{thm:equivalence 1} and Theorem~\ref{thm:equivalence 2}.  Note that  the results in Theorem~\ref{thm:equivalence 1} and Theorem~\ref{thm:equivalence 2} hold  if an additional drift term $F$ is incorporated.

\begin{theorem}
\label{thm:equi_SKT}
The {local} pathwise mild solution {$(u, \tau)$} of~\eqref{eq:SKT_abs} is also a weak solution. More precisely, for $t \in [0, \tau)$, the following relation holds $\mathbb{P}$-a.s.
\begin{align}
\label{eq:equiSKT}
\langle u(t), x^{\ast}\rangle &= \left\langle u^0,x^{\ast}\right\rangle - \int_{0}^{t}\langle B(u(s)) \nabla u(s) ,\nabla x^\ast\rangle \ds - \int_{0}^{t}\langle \Gamma(u(s)), x^\ast\rangle \ds  \\
& \qquad +  \int_{0}^{t} \left\langle F(u(s)), x^{*}\right\rangle\textnormal{d}s + \left\langle \int_{0}^{t}\sigma(u(s))\dW_{s}, x^{*}\right\rangle, \nonumber
\end{align}
where $x^\ast \in \mathcal{D}^\ast_u$.
\end{theorem}

\begin{proof}[Proof]
 By Theorem~\ref{thm:equivalence 1} part (1), we infer that the pathwise mild solution $u=(u_1,u_2)$ satisfies the weak formulation {for $t \in [0,\tau)$}
\begin{align}
\label{eq:equiSKT1}
\langle u(t), \varphi(t)\rangle & = \left\langle u^0,\varphi(0)\right\rangle - \int_{0}^{t}\langle B(u(s)) \nabla u(s) ,\nabla \varphi(s) \rangle \ds - \int_{0}^{t}\langle \Gamma(u(s)), \varphi(s)\rangle \ds \\ 
 &\quad + \int_{0}^{t}\left\langle u(s), \varphi^{\prime}(s)\right\rangle \ds - \int_{0}^{t}\left\langle \int_{0}^{s}\sigma(u(\tau))\dW_{\tau}, \varphi^{\prime}(s)\right\rangle \ds \nonumber \\ 
 & \quad + \int_{0}^{t} \left\langle F(u(s)), {\varphi(s)}\right\rangle\textnormal{d}s +\left\langle \int_{0}^{t}\sigma(u(s))\dW_{s}, \varphi(t)\right\rangle,\nonumber 
\end{align}
for every time-dependent test function $\varphi\in L^{0}\left(\Omega ; C^{1}\left([0, t]; \mathcal{D}_{u}^{*})\right)\right)$. Now,  Theorem \ref{thm:equivalence 2} part (1) entails the usual weak formulation \eqref{eq:equiSKT}.  
\end{proof}

\begin{remark}
Theorem~\ref{thm:equi_SKT} also provides a  regularity result for the weak solution, i.e. $u\in L^0(\Omega; C^{\delta}([0,\tau);Y))$.
\end{remark}

\begin{example}
We let $d\geq 1$ and consider a quasilinear parabolic stochastic partial differential equation on a $d$-dimensional domain $\dom\subset\mathbb{R}^d$ with smooth boundary $\partial\dom$ of the form
\begin{equation}\label{eq:SPDEHof}
    \begin{cases}
    &\mathrm{d} u=\operatorname{div}(B(u) \nabla u) \mathrm{d} t+\sigma(u) \textnormal{d} W_t, \quad \mbox{ in }  (0,T)\times\dom\\
   & u= 0 \quad \quad\quad \hspace*{42 mm} \mbox{ on } [0,T]\times\partial\dom\\
    &u(0,x) = u^0(x), \qquad \hspace*{29 mm}x \in \dom, 
    \end{cases}
\end{equation}
 where $(W_t)_{t\in[0,T]}$ is a cylindrical Wiener process taking values in a Hilbert space $H \supset X: = L^{2}\left(\dom\right)$.
\end{example}

Such equations have been extensively studied in the literature, see \cite{DMH15, DHV16, HoZh17}, under the following assumptions on the coefficients $B$ and $\sigma$:
\begin{assumption}
\label{ass:SPDEHof}
\begin{itemize}
    \item[\textnormal{(1)}] The coefficients $B: \mathbb{R} \rightarrow \mathbb{R}^{d \times d}$ are nonlinear functions, such that the diffusion matrix $B = \left(B_{i j}\right)_{i, j=1}^{d}$ is of class $C_{b}^{1}$, symmetric, uniformly positive definite and bounded, i.e. there exist constants $\kappa, C >0$ such that 
    \begin{equation}
    \label{eq:bdd_diff_mat}
    \kappa \mathrm{I} \leq B \leq C \mathrm{I}.
    \end{equation}
    \item[\textnormal{(2)}] For each $u \in X$ we consider a mapping $\sigma(u): H \rightarrow X$ defined by 
    \[\sigma(u) {e}_{k}=\sigma_{k}(\cdot,u(\cdot)),\]
    where $\sigma_{k}\in C(\dom\times\mathbb{R})$. We further suppose that $\sigma$ satisfies  usual Lipschitz and linear growth conditions, i.e.
\[
\begin{split}
&\sum_{k \in \N}\left|\sigma_{k}\left(x,\xi_{1}\right)-\sigma_{k}\left(x,\xi_{2}\right)\right|^{2} \leq C\left|\xi_{1} - \xi_{2}\right|^{2}, \qquad \forall\, x\in\dom,  \xi_1,\, \xi_2 \in \R,\\
&\sum_{k \in \N}\left|\sigma_{k}(x,\xi)\right|^{2} \leq C\left(1+|\xi|^{2}\right), \qquad \qquad \qquad \quad \;\forall\,x\in\dom, \xi \in \mathbb{R}.
\end{split}
\]
\end{itemize}
\end{assumption}

In particular, these  assumptions imply that $\sigma$ maps $X$ to $\mathcal{L}_{2}(H, X)$. Thus, given a predictable process $u$ that belongs to $L^{2}\left(\Omega, L^{2}(0, T ; X)\right)$, the stochastic integral is a well-defined $X$-valued process.

The previous assumptions on $\sigma$ can be relaxed and an additional regular drift term can be incorporated~\cite{DMH15, DHV16, HoZh17}. An example of such a drift term is given by
 $\diver(F(u))$ , where \[F=\left(F_{1}, \ldots, F_{d}\right): \mathbb{R} \longrightarrow \mathbb{R}^{d}\]
is continuously differentiable with bounded derivatives. 
{
\begin{remark}
In contrast to Example \ref{eq:SKT}, the diffusion matrix $B$ additionally satisfies the boundedness assumption  \eqref{eq:bdd_diff_mat}.
\end{remark}
}
Next, we give the definition of a weak solution of \eqref{eq:SPDEHof}.

\begin{definition}
An $\left(\mathcal{F}_{t} \right)_{t\in[0,T]}$-adapted, $X$-valued continuous process $(u(t))_{t\in[0,T]}$ is called a weak solution for \eqref{eq:SPDEHof} if for any $\varphi \in C^{\infty}_0\left(\dom\right)$, the following identity holds for $t\in[0,T]$, $\mathbb{P}$-a.s.
      \begin{align}
    &\langle u(t), \varphi\rangle  =  \left\langle u^{0}, \varphi\right\rangle- \int_{0}^{t}\langle B(u(s)) \nabla u(s), \nabla \varphi\rangle\,\ds + \int_{0}^{t}\langle\sigma(u(s))\,\dW_s, \varphi\rangle . \nonumber
    \end{align}
\end{definition}

Under the previous assumptions on $B$ and $\sigma$, together with suitable regularity conditions on the initial data, the existence of a weak solution was established in~\cite{DMH15,DHV16}.

Moreover, assuming higher spatial regularity on $\sigma$, the regularity of this weak solution can be improved~\cite[Theorem~2.6]{DMH15}. For the convenience of the reader we indicate this statement. To this aim we let $\eta>0$, set $D_T:=[0,T]\times\dom $ and consider the H\"older space $C^{\eta/2,\eta}(D_T)$ with different time and space regularity, endowed with the norm
\begin{align*}
 ||f||_{C^{\eta/2,\eta}(D_T)}=\sup\limits_{(t,x)\in D_T} |f(t,x)| + \sup\limits_{(t,x)\neq (s,y)\in D_T} \frac{|f(t,x)-f(s,y)|}{\max\{|t-s|^{\eta/2} +|x-y|^{\eta}\} }.
\end{align*}

\begin{theorem}\label{thm:hofregularity}
Assume that
\begin{itemize}
    \item $u^{0} \in L^{m}\left(\Omega ; C^{\iota}(\bar{\dom})\right)$ for some $\iota>0$ and all $m \in[2, \infty),$ and $u^{0}=0$ on $\partial \dom$ a.s.
    \item $\|\sigma(u)\|_{\mathcal{L}_{2}\left(H, H_{0}^{1}(\dom)\right)} \leq C\left(1+\|u\|_{H_{0}^{1}(\dom)}\right)$. 
\end{itemize}
Then the weak solution $u$ of~\eqref{eq:SPDEHof}   belongs to $L^{m}\left(\Omega ; C^{\eta / 2, \eta}(\overline{D_{T}})\right)$, for all $m\in[2,\infty)$. 
\end{theorem}

Now, we are ready to state the main result for this {example}, based on Theorem \ref{thm:hofregularity}.

\begin{theorem}
  {Let assumptions of Theorem~\ref{thm:hofregularity} and \textbf{Assumption~\ref{ass:SPDEHof}} hold. Then, there exists a pathwise mild solution of \eqref{eq:SPDEHof}. More precisely, there exists an}  $(\mathcal{F}_t)_{t\in[0,T]}$-adapted process $u$ such that for $t \in[0,T]$, $\mathbb{P}$-a.s.
 \begin{equation}
     \label{eq:hofequi}
     u(t) = U^u(t,0) u^0 + U^u(t,0) \int_0^t \sigma(u(s))\dW_s - \int_0^t U^u(t,s) A_u(s)\int_s^t \sigma(u(\tau))\dW_\tau\,\ds,
 \end{equation}
 where $A_u u := \diver(B(u) \nabla u)$ and $U^u(\cdot,\cdot)$ is the evolution family generated by $A_u$. Moreover, $u \in L^0(\Omega; \mathcal{B}([0,T]; H^1(\dom))) \cap L^0(\Omega; C^\delta([0,T]; L^2(\dom)))$.
\end{theorem}

\begin{proof}
 Under \textbf{Assumption~\ref{ass:SPDEHof}}, the existence of a weak solution $u$
 \[u \in  L^{2}\left(\Omega ; C\left([0, T] ; L^{2}(\dom)\right) \cap L^{2}\left(0, T ; H^1(\dom)\right)\right)\]
 was established in~\cite{DMH15,DHV16}. The bilinear form $a$ for $u \in H^1(\dom)$ given by
 \[a(u; \mathrm{v}, \mathrm{w}):= \langle B(u) \nabla \mathrm{v}, \nabla \mathrm{w}\rangle, \qquad \mathrm{v}, \mathrm{w} \in H^1(\dom)\]
 satisfy conditions \eqref{blf:coercivity} and \eqref{blf:continuity} with $Z := H^{1}(\dom)$, $X = Y := L^{2}(\dom)$, thanks to \eqref{eq:bdd_diff_mat}. Therefore, \ref{ass:AT1} and \ref{ass:AT2} from \textbf{Assumption~\ref{ass:AT}} hold. The Lipschitz assumption on $B$ entails that \ref{ass:AT3} holds too. Similarly, it can be shown that the conditions \ref{ass:AT1ad}--\ref{ass:AT3ad} of \textbf{Assumption~\ref{ass:ATadjoint}} are satisfied by the adjoint $A^\ast_u$, see \cite[Chapter~1.8.2]{Ya10} for details. We are left to verify \ref{ass:AT4}. To this aim, we require to show that $u\in C^{\delta}([0,T];L^2(\mathcal{O}))$ for some $\delta > 0$. The weak solution $u \in L^{m}\left(\Omega ; C^{\eta / 2, \eta}(\overline{D_{T}})\right)$, for all $m\in[2,\infty)$, by Theorem~\ref{thm:hofregularity}. Regarding this along with the following equivalent norm \cite{Kr96} on $C^{\eta/2,\eta}(D_{T})$ 
\[
\|u\|_{\eta / 2, \eta, D_T} = \sup _{(t, x),(s, x) \in D_{T}} \frac{|u(t, x)-u(s, x)|}{|t-s|^{\eta / 2}}+\underset{(t, x),(t, y) \in D_{T}}{\sup } \frac{|u(t, x)-u(t, y)|}{|x-y|^{\eta}}
\] 
and boundedness of the domain $\dom$, we conclude that $u \in C^\delta([0,T]; L^2(\dom))$ for $\delta = \eta/2$.

\noindent Hence, we can infer from Theorem~\ref{thm:semigp} that the operator $A_u u = \diver(B(u) \nabla u)$ generates an evolution family $U^u$. According to Theorem~\ref{thm:equivalence 2} part (2), this {evolution family along with the weak solution $u$} satisfies \eqref{eq:hofequi}. Consequently, we obtain that $u$ is a pathwise mild solution of \eqref{eq:SPDEHof}.
\end{proof}

{\begin{remark}\label{rem.sigma_holder}
Note that the  regularity on $\sigma$ assumed in Theorem~\ref{thm:hofregularity} is necessary in order to obtain the H\"older regularity of the solution.
\end{remark}
}

\appendix
\section{Fractional powers of sectorial operators}
\label{sec:frac_pow_dom}
Let $X$ be a Banach space with norm $\|\cdot\|_X$ and $A$ be a linear sectorial operator of $X$ with angle $ 0 \le \vartheta_A < \pi$. As before an open sectorial domain $\Sigma_\vartheta$ for $\vartheta_A < \vartheta < \pi$ is given by
\[\Sigma_\vartheta = \left\{\lambda \in \mathbb{C};\, |\operatorname{arg} \lambda| < \vartheta\right\}, \quad \vartheta_A < \vartheta < \pi.\]
We define, for each complex number $z$ with $\operatorname{Re}z >0,$ the bounded linear operator
$$
A^{-z}=\frac{1}{2 \pi i} \int_{\Gamma} \lambda^{-z}(\lambda-A)^{-1} \textnormal{d} \lambda,
$$
using the Dunford integral in $\mathcal{L}(X)$, where $\Gamma$ is the contour surrounding the spectrum $\sigma(A)$, running counterclockwise in $\mathbb{C}\setminus(\infty, 0] \cap \rho(A)$. If $z=n\in \mathbb{N}$ it can be shown, \cite[Chapter~2.7.1]{Ya10}, that this definition coincides with the standard definition of {$A^{-n} = \left(A^n\right)^{-1}$}. 

$A^{-z}$ is an analytic function for $\operatorname{Re} z > 0$ with values in $\mathcal{L}(X)$. The following theorem \cite[Theorem 2.21]{Ya10} is concerned with the convergence of $A^{-z}$ as $z \to 0$.

\begin{theorem}
For any 0 $<\phi<\frac{\pi}{2},$ as $z \rightarrow 0$ with $z \in \overline{\Sigma_{\vartheta}}\setminus\{0\}$, $A^{-z}$ converges to $\mathrm{Id}$ strongly on $X$.
\end{theorem}

It also holds that $A^{-z}$ satisfies the law of exponent, i.e. \[A^{-z}A^{-z^{\prime}} = A^{-\left(z+z^{\prime}\right)}, \quad \operatorname{Re} z > 0,\, \operatorname{Re} z^{\prime} > 0,\] 
which leads to the following theorem, see \cite[Theorem 2.22]{Ya10}.

\begin{theorem}
The $\mathcal{L}(X)$-valued function $A^{-z}$ is an analytic semigroup defined in the half-plane $\{z \in \mathbb{C} ; \operatorname{Re} z>0\}$.
\end{theorem}

The fractional power $A^\alpha$, for every real number $-\infty < \alpha < \infty$ is defined, see \cite[Chapter~2.7.2]{Ya10} for details. The following theorem lists some of the properties of the fractional power $A^\alpha$, \textit{cf.} \cite[Theorem 2.23]{Ya10}.

\begin{theorem}
Let $A$ be a sectorial operator on $X$ with angle $\vartheta_A$. Then
\begin{itemize}
    \item[\textnormal{(1)}] for $-\infty<\alpha<0$, $A^{\alpha}$ are bounded operators on $X$. $A^{0} = \mathrm{Id}$ on $X$ and $A^{\alpha}$ are densely defined, closed linear operators of $X$ for $\alpha > 0$.
    
    \item[\textnormal{(2)}] Let $0 \leq \alpha_{1}<\alpha_{2}<\infty$, then $D\left(A^{\alpha_{2}}\right) \subset D\left(A^{\alpha_{1}}\right)$.
    
    \item[\textnormal{(3)}] $A^\alpha$ satisfies the law of exponent, i.e. \[A^{\alpha}A^{\beta} = A^{\beta}A^{\alpha} = A^{\alpha+\beta},\quad  -\infty < \alpha, \beta < \infty.\]
    
    \item[\textnormal{(4)}] For $0 < \alpha < 1$, $A^\alpha$ is a sectorial operator on $X$ with angle $\le \alpha \vartheta_A$.
\end{itemize}
\end{theorem}

Let $0 < \alpha < 1$ and $A^{\alpha}$ be the fractional powers of $A$. Let $M_\pi$ be the constant appearing in the assumption \ref{ass:AT2} with angle $\vartheta = \pi$. Let us introduce the spaces
$$
D_{\alpha}(A)=\left\{\mathrm{v} \in X \colon \sup _{0<\rho<\infty} \rho^{\alpha}\|A(\rho+A)^{-1} \mathrm{v}\|_{X} < \infty\right\}, \quad 0 \leq \alpha \leq 1.
$$
$D_{\alpha}(A)$ are normed spaces, equipped with the norms
$$
\|\mathrm{v}\|_{D_{\alpha}(A)}=\sup _{0<\rho<\infty} \rho^{\alpha}\|A(\rho+A)^{-1} \mathrm{v}\|_{X}.
$$

In the following we compare the domain of fractional powers of the operator $A$ to that of $D_\alpha(A)$ \cite[Theorem 2.24]{Ya10}.

\begin{theorem}
For any $0<\alpha<1, D\left(A^{\alpha}\right) \subset D_{\alpha}(A),$ and the estimate
$$
\|\mathrm{v}\|_{D_{\alpha}(A)} \leq C\left(1 + M_{\pi}\right)^{2}\|A^{\alpha} \mathrm{v}\|_{X}, \quad \mathrm{v} \in D\left(A^{\alpha}\right)
$$
holds true. Conversely, for any $0<\alpha<\alpha^{\prime}<1, D_{\alpha^{\prime}}(A) \subset D\left(A^{\alpha}\right),$ and the estimate
$$
\|A^{\alpha} \mathrm{v}\|_{X} \leq C_{\alpha, \alpha^{\prime}}\left[\|\mathrm{v}\|_{D_{\alpha^{\prime}}(A)}+\left( 1 + M_{\pi}\right)\|\mathrm{v}\|_{X}\right], \quad \mathrm{v} \in D_{\alpha^{\prime}}(A),
$$
holds true.
\end{theorem}

Next, we compare domains of fractional powers of two sectorial operators $A$ and $B$ of $X$ for which $D(A) \subset D(B)$ continuously, i.e. there exists a constant $C > 0$ such that
$$
\|B \mathrm{v}\|_{X} \leq C\|A \mathrm{v}\|_{X}, \quad \mathrm{v} \in D(A).
$$

\begin{theorem}\cite[Theorem 2.25]{Ya10}
Let $A$ and $B$ be two sectorial operators of $X$ satisfying the above relationship between their domains, as well as assumptions \ref{ass:AT1} and \ref{ass:AT2}. Then, for any $0<\alpha<\alpha^{\prime}<1$, $D(A^{\alpha^{\prime}}) \subset D(B^{\alpha})$
and the estimate
$$
\|B^{\alpha} \mathrm{v}\|_{X} \leq C_{\alpha^{\prime}, \alpha}\|A^{\alpha^{\prime}} \mathrm{v}\|_{X}, \quad \mathrm{v} \in D(A^{\alpha^{\prime}})
$$
holds true, where $C_{\alpha, \alpha^{\prime}}>0$ is determined by $\alpha, \alpha^{\prime}, M_{\pi}$ and $C$.
\end{theorem}
 

\end{document}